\documentclass[11pt]{article}
\usepackage[margin=1in]{geometry}

\usepackage{graphicx} % for pdf, bitmapped graphics files
\usepackage{amsmath} % assumes amsmath package installed
\usepackage{amssymb}  % assumes amsmath package installed
\usepackage{amsthm}

\usepackage{scalefnt}
\usepackage{hyperref} % Amirali: I commented this package out because after 22 equations, it wasn't
\usepackage{lscape}  %AAA: adding this for a large matrix in the appendix.
\usepackage{color}
\usepackage{subfigure}

\usepackage{comment}

\makeatletter
\renewcommand*\env@matrix[1][c]{\hskip -\arraycolsep
  \let\@ifnextchar\new@ifnextchar
  \array{*\c@MaxMatrixCols #1}}
\makeatother

\theoremstyle{plain}  % Bold name, italics font
\newtheorem{theorem}{Theorem}[section]
\newtheorem{lemma}[theorem]{Lemma}
\newtheorem{proposition}[theorem]{Proposition}
\newtheorem{corollary}[theorem]{Corollary}
\newtheorem{definition}[theorem]{Definition}

\theoremstyle{definition}

\theoremstyle{remark} % italics name, roman font
\newtheorem{example}{Example}[section]
\newtheorem{remark}{Remark}[section]

\definecolor{Pink}{RGB}{225,0,127}
\long\def\aaa#1{{#1}}

\title{\LARGE \bf Stability of Polynomial Differential Equations: Complexity and Converse Lyapunov Questions}
 \author{Amir Ali Ahmadi and Pablo A. Parrilo \thanks{Amir Ali Ahmadi is a Goldstine Fellow at the Department of Business Analytics and Mathematical Sciences of the IBM Watson Research Center. Pablo A. Parrilo is with the Laboratory for Information and Decision Systems, Department of Electrical Engineering and Computer Science, Massachusetts Institute of Technology.\newline Email:
\{\texttt{a\_a\_a}, \texttt{parrilo}\}\texttt{@mit.edu}.  }
\thanks{This research was partially supported by the NSF Focused
Research Group Grant on Semidefinite Optimization and Convex
Algebraic Geometry DMS-0757207.} }
\begin{document}
\date{}
\maketitle

%\thispagestyle{empty}
%\pagestyle{empty}

%%%%%%%%%%%%%%%%%%%%%%%%%%%%%%%%%%%%%%%%%%%%%%%%%%%%%%%%%%%%%%%%%%%%%%%%%%%%%%%%
\begin{abstract}
Stability analysis of polynomial differential equations is a central topic in nonlinear dynamics and control which in recent years has undergone major \emph{algorithmic} developments due to advances in optimization theory. Notably, the last decade has seen a widespread interest in the use of sum of squares (sos) based semidefinite programs that can automatically find polynomial Lyapunov functions and produce explicit certificates of stability. However, despite their popularity, the converse question of whether such algebraic, efficiently constructable certificates of stability always exist has remained elusive. Naturally, an algorithmic question of this nature is closely intertwined with the fundamental computational complexity of proving stability. In this paper, we make a number of contributions to the questions of (i) complexity of deciding stability, (ii) existence of polynomial Lyapunov functions, and (iii) existence of sos Lyapunov functions.

(i) We show that deciding local or global asymptotic stability of \aaa{cubic} vector fields is strongly NP-hard. Simple variations of our proof are shown to imply strong NP-hardness of several other decision problems: testing local attractivity of an equilibrium point, stability of an equilibrium point in the sense of Lyapunov, \aaa{invariance of the unit ball}, boundedness of trajectories, convergence of all trajectories in a ball to a given equilibrium point, existence of a quadratic Lyapunov function, local collision avoidance, and existence of a stabilizing control law. 

%As a byproduct of our NP reduction, we obtain a Lyapunov-inspired technique for proving positivity of forms which can %outperform the standard sum of squares test.

(ii) We present a simple, explicit example of a globally asymptotically stable quadratic vector field on the plane which does not admit a polynomial Lyapunov function (joint work with M. Krstic and presented here without proof). For the subclass of homogeneous vector fields, we conjecture that asymptotic stability implies existence of a polynomial Lyapunov function, but show that the minimum degree of such a Lyapunov function can be arbitrarily large even for vector fields in fixed dimension and degree. For the same class of vector fields, we further establish that there is no monotonicity in the degree of polynomial Lyapunov functions.

%; i.e., a homogeneous cubic vector field with no homogeneous polynomial Lyapunov function of some degree $d$ can very well %have a homogeneous polynomial Lyapunov function of degree less than $d$.

(iii) We show via an explicit counterexample that if the degree of the polynomial Lyapunov function is fixed, then sos programming \aaa{may} fail to find a valid Lyapunov function even though one exists. On the other hand, if the degree is allowed to increase, we prove that existence of a polynomial Lyapunov function for a planar or a homogeneous vector field implies existence of a polynomial Lyapunov function that is sos and that the negative of its derivative is also sos. 

%This result is extended to prove that asymptotic stability of switched linear systems can always be proven with sum of squares Lyapunov functions. Finally, we show that for the latter class of systems (both in discrete and continuous time), if the negative of the derivative of a Lyapunov function is sos, then the Lyapunov function itself is automatically sos.

\end{abstract}

%%%%%%%%%%%%%%%%%%%%%%%%%%%%%%%%%%%%%%%%%%%%%%%%%%%%%%%%%%%%%%%%%%%%%%%%%%%%%%%%

%%%\begin{itemize}
%%%\item The introduction section needs to be rewritten.
%%%\item This paper is based on parts of~\cite{AAA_PP_CDC11_converseSOS_Lyap}, \cite{AAA_MK_PP_CDC11_no_Poly_Lyap}, \cite{AAA_ACC12_Cubic_Difficult}, \cite{ACC13_complexity_CT10}, \cite[Chap. 4]{AAA_PhD}.
%%%\end{itemize}

\newpage

\section{Introduction}

The diversity of application domains in science and engineering where complex dynamical systems are modelled as polynomial differential equations is unsurpassed. Dynamics of population growth in ecology, prices and business cycles in economics, chemical reactions in cell biology, spread of epidemics in network science, and motion of a wide range of electromechanical systems in control engineering, are but few examples of areas where slight deviations from the traditional ``simplifying assumptions of linearity'' leave us with polynomial vector fields. Polynomial systems also enjoy much added interest stemming from the fact that various other types of dynamical systems are often approximated by their Taylor expansions of some order around equilibrium points. A prime example of this is in the field of robotics where equations of motion (described by the manipulator equations) give rise to trigonometric differential equations which are then commonly approximated by polynomials for analysis and control. Aside from practice, polynomial vector fields have also always been at the heart of diverse branches of mathematics: Some of the earliest examples of chaos (e.g. the Lorenz attractor) arise from simple low degree polynomial vector fields~\cite{Lorenz_chaos}; the still-open Hilbert's 16th problem concerns polynomial vector fields on the plane~\cite{Hilbert16}; Shannon's General Purpose Analog Computer is equivalent in simulation power to systems of polynomial differential equations~\cite{Polynomial_Analog_Comp}; and the list goes on.

In many application areas, the question of interest is often not to solve for particular solutions of these differential equation numerically, but rather understand certain qualitative properties of the trajectories as a whole. Among the different qualitative questions one can ask, the question of \emph{stability} of equilibrium points is arguably the most ubiquitous. Will deviations from equilibrium prices in the market be drawn back to the equilibrium? Will the epidemic die out or become widespread? We are concerned for the most of part of this paper with this stability question, which we study from an algorithmic viewpoint. Our goal is to shed light on the complexity of this question and to understand the power/limitations of some of the most promising optimization-based computational techniques currently available for tackling this problem.

Almost universally, the study of stability in dynamical systems leads
to Lyapunov's second method or one of its many variants. An
outgrowth of Lyapunov's 1892 doctoral
dissertation~\cite{PhD:Lyapunov}, Lyapunov's second method tells
us, roughly speaking, that if we succeed in finding a
\emph{Lyapunov function}---an energy-like function of the state
that decreases along trajectories---then we have proven that the
dynamical system in question is stable. In the mid 1900s, a series
of \emph{converse Lyapunov theorems} were developed which
established that any stable system indeed has a Lyapunov function
(see~\cite[Chap. 6]{Hahn_stability_book} for an overview).
Although this is encouraging, except for the simplest classes of
systems such as linear systems, converse Lyapunov theorems do not
provide much practical insight into how one may go about finding a
Lyapunov function.

%Moreover, the Lyapunov functions constructed in these classical
%theorems often belong to an infinite dimensional space of
%functions (e.g. the space of all continuously differentiable
%functions) and because of this reason the theorems do not really
%restrict the search space in

In the last few decades however, advances in the theory and
practice of convex optimization and in particular semidefinite
programming (SDP) have rejuvenated Lyapunov theory. The approach
has been to parameterize a class of Lyapunov functions with
restricted complexity (e.g., quadratics, pointwise maximum of
quadratics, polynomials, etc.) and then pose the search for a
Lyapunov function as a convex feasibility problem. A widely
popular example of this framework which we study in
this paper is the method of sum of squares (sos) Lyapunov
functions~\cite{PhD:Parrilo},\cite{PabloLyap}. Expanding on the
concept of sum of squares decomposition of polynomials and its relation to semidefinite programming, this technique allows one to formulate SDPs that automatically search for polynomial Lyapunov functions for polynomial dynamical systems. Over the last decade, the applicability of sum of squares Lyapunov
functions has been explored in many directions and numerous extensions have been developed to tackle a wide range
of problems in systems and control. We refer the reader to the by
no means exhaustive list of works
\cite{PositivePolyInControlBook},
\cite{AutControlSpecial_PositivePolys}, \cite{ControlAppsSOS},
\cite{PraP03}, \cite{Pablo_Rantzer_synthesis},
\cite{Chest.et.al.sos.robust.stability},
\cite{AAA_PP_CDC08_non_monotonic}, \cite{Erin_Pablo_Contraction},
and references therein. Despite the wealth of research in this area, however, the converse question
of whether a proof of stability via sum of squares Lyapunov functions is always possible has remained elusive.

Our paper speaks to the premise that an algorithmic approach to Lyapunov theory naturally calls for new
converse theorems. Indeed, classical converse Lyapunov theorems
only guarantee existence of Lyapunov functions within very broad
classes of functions (e.g. the class of continuously
differentiable functions) that are a priori not amenable to
computation. So there is the need to know whether Lyapunov
functions belonging to certain more restricted classes of
functions that can be computationally searched over also exist.
For example, do stable polynomial systems admit Lyapunov functions
that are polynomial? What about polynomial functions that can be
found with sum of squares techniques? 

As one would expect, questions of this nature are intrinsically related to the computational complexity of deciding stability. As an example, consider the following fundamental question which to the best of our knowledge is open: Given a polynomial vector field (with rational coefficients), is there an algorithm that can decide whether the origin is a (locally or globally) asymptotically stable equilibrium point? A well-known conjecture of Arnold~\cite{Arnold_Problems_for_Math} (see Section~\ref{sec:complexity.cubic.vec.field}) states that there cannot be such an algorithm; i.e., that the problem is \emph{undecidable}. Suppose now that one could prove that (local or global) asymptotic stability of polynomial vector fields implies existence of a polynomial Lyapunov function, together with a \emph{computable upper bound} on its degree. Such a result would imply that the question of stability is \emph{decidable}. Indeed, given a polynomial vector field and an integer $d$, one could e.g. use the quantifier elimination theory of Tarski and Seidenberg~\cite{Tarski_quantifier_elim},~\cite{Seidenberg_quantifier_elim} to test, in finite time, whether the vector field admits a polynomial Lyapunov function of degree $d$.

Just as a proof of undecidability for stability would rule out existence of polynomial Lyapunov functions of a priori known degree, a hardness result on \emph{time complexity} (e.g. an NP-hardness result) would inform us that we should not always expect to find ``efficiently constructable'' Lyapunov functions (e.g. those based on convex optimization). In fact, much of the motivation behind sum of squares Lyapunov functions as a replacement for polynomial Lyapunov functions is computational efficiency. Even though polynomial functions are finitely parameterized, the computational problem of finding a polynomial that satisfies the Lyapunov inequalities (nonnegativity of the function and nonpositivity of its derivative) is intractable. (This is due to the fact that even the task of checking if a given quartic polynomial is nonnegative is NP-hard~\cite{nonnegativity_NP_hard},~\cite{Minimize_poly_Pablo}). By contrast, when one replaces the two Lyapunov inequalities by the requirement of having a sum of squares decomposition, the search for this (sum of squ\aaa{a}res) Lyapunov function becomes a semidefinite program, which can be solved efficiently e.g. using interior point methods~\cite{VaB:96}.

It is important to also keep in mind the distinctions between establishing hardness results for stability versus answering questions on existence of Lyapunov functions. On one hand, complexity result\aaa{s} have much stronger implications in the sense that they rule out the possibility of \emph{any} (efficient) algorithms for deciding stability, not just those that may be based on a search for a particular family of Lyapunov functions, or based on Lyapunov theory to begin with for that matter. On the other hand, complexity results only imply that ``hard instances'' of the problem should exist asymptotically (when the number of variables is large), often without giving rise to an explicit example or determining whether such instances also exist within reasonable dimensions. For these reasons, a separate study of both questions is granted, but the connections between the two are always to be noted.

\subsection{Contributions and organization of the paper} \label{subsec:contributions}
We begin this paper by a short section on preliminaries (Section~\ref{sec:prelims}), where we formally define some basic notions of interest, such as sum of squares Lyapunov functions, homogeneous vector fields, etc. The next three sections respectively include our main contributions on (i) complexity of testing stability, (ii) existence of polynomial Lyapunov functions, and (iii) existence of sos Lyapunov functions. These contributions are outlined below. We end the paper in Section~\ref{sec:summary.future.work} with a list of open problems.

\paragraph{Complexity results (Section~\ref{sec:complexity.cubic.vec.field}).} After a brief account of a conjecture of Arnold on undecidability of testing stability, we present the following NP-hardness results:
\begin{itemize}
\item We prove that deciding (local or global) asymptotic stability of cubic vector fields is strongly NP-hard even when the vector fields are restricted to be homogeneous (Theorem~\ref{thm:asym.stability.nphard}). Degree three is the minimum degree possible for such a hardness result as far as homogeneous systems are concerned (Lemma~\ref{lem:linear.in.P}). The main challenge in establishing this result is to find a way around the seemingly hopeless task of relating solutions of a combinatorial problem to trajectories of cubic differential equations, to which we do not even have explicit access.

\item By modifying our reduction appropriately, we further prove that the following decision problems which arise in numerous areas of systems and control theory are also strongly NP-hard (Theorem~\ref{thm:poly.hardness.results}):

\begin{itemize}
\item testing local attractivity of an equilibrium point,
\item testing stability of an equilibrium point in the sense of Lyapunov,
\item testing boundedness of trajectories,
\item testing convergence of all trajectories in a ball to a given equilibrium point,
\item testing existence of a quadratic Lyapunov function,
\item testing local collision avoidance, 
\item testing existence of a stabilizing control law,
\item \aaa{testing invariance of the unit ball,}
\item \aaa{testing invariance of a basic semialgebraic set under linear dynamics.}
\end{itemize}

%\item As a byproduct of our main reduction, we obtain a Lyapunov-inspired technique for proving positivity of forms (Corollary~\ref{cor:positivity.forms.Lyap}). This approach is at least as powerful as checking for a sum of squares decomposition. We show that it can be strictly better (Example~\ref{ex:positivity.thru.Lyap.beat.Motzkin}).

\end{itemize}

\paragraph{(Non)-existence of polynomial Lyapunov functions (Section~\ref{sec:no.poly.Lyap}).} In this section we study whether asymptotic stability implies existence of a polynomial Lyapunov function, and if so whether an upper bound on the degree of such a Lyapunov function can be given.

\begin{itemize}
\item We present a simple, explicit example of a quadratic differential equation in two variables and with rational coefficients that is globally asymptotically stable but does not have a polynomial Lyapunov function of any degree (Theorem~\ref{thm:GAS.poly.no.poly.Lyap}). This is joint work with Miroslav Krstic~\cite{AAA_MK_PP_CDC11_no_Poly_Lyap} and is presented here without proof.

\item Unlike the general case, we conjecture that for the subclass of homogeneous vector fields (see Section~\ref{sec:prelims} for a definition) existence of a (homogeneous) \emph{polynomial} Lyapunov function is necessary and sufficient for asymptotic stability. This class of vector fields has been extensively studied in the literature on nonlinear control~\cite{Stability_homog_poly_ODE}, \cite{Stabilize_Homog}, \cite{homog.feedback}, \cite{Baillieul_Homog_geometry}, \cite{Cubic_Homog_Planar}, \cite{HomogHomog}, \cite{homog.systems}. Although a polynomial Lyapunov function is conjectured to always exist, we show that the degree of such a Lyapunov function cannot be bounded as a function of the degree and the dimension of the vector field only. (This is in contrast with stable linear systems which always admit quadratic Lyapunov functions.) We prove this by presenting a family of globally asymptotically stable cubic homogeneous vector fields in two variables for which the minimum degree of a polynomial Lyapunov function can be arbitrarily large (Theorem~\ref{thm:no.finite.bound}).

\item For homogeneous vector fields of degree as low as three, we show that there is no monotonicity in the degree of polynomial Lyapunov functions that prove asymptotic stability; i.e., a homogeneous cubic vector field with no homogeneous
polynomial Lyapunov function of some degree $d$ can very well have a homogeneous polynomial Lyapunov function of degree
less than $d$ (Theorem~\ref{thm:no.monotonicity}).

\end{itemize}

\paragraph{(Non)-existence of sum of squares Lyapunov functions (Section~\ref{sec:(non)-existence.sos.lyap}).} The question of interest in this section is to investigate whether existence of a polynomial Lyapunov function for a polynomial vector field implies existence of a sum of squares Lyapunov function (see Definition~\ref{def:sos.Lyap}). Naturally this question comes in two variants: (i) does an sos Lyapunov function of the same degree as the original polynomial Lyapunov function exist? (ii) does an sos Lyapunov function of possibly higher degree exist? Both questions are studied in this section. We first state a well-known result of Hilbert on the gap between nonnegative and sum of squares polynomials and explain why our desired result does not follow from Hilbert's work. Then, we present the following results:

\begin{itemize}

\item In Subsection~\ref{subsec:the.counterexample}, we show via an explicit counterexample that existence of a polynomial Lyapunov function does not imply existence of a sum of squares Lyapunov function of the same degree. This is done by proving infeasibility of a certain semidefinite program.

\item By contrast, we show in Subsection~\ref{subsec:converse.sos.results} that existence of a polynomial Lyapunov function always implies existence of a sum of squares Lyapunov function of possibly higher degree in the case where the vector field is homogeneous (Theorem~\ref{thm:poly.lyap.then.sos.lyap}) or when it is planar and an additional mild assumption is met (Theorem~\ref{thm:poly.lyap.then.sos.lyap.PLANAR}). The proofs of these two theorems are quite simple and rely on powerful
Positivstellensatz results due to Scheiderer (Theorems~\ref{thm:claus} and~\ref{thm:claus.3vars}). 

%\item In Section~\ref{sec:extension.to.switched.sys}, we extend these results to derive a converse sos Lyapunov theorem for robust stability of switched linear systems. It is shown that if such a system is stable under arbitrary switching, then it admits a common polynomial Lyapunov function that is sos and that the negative of its derivative is also sos (Theorem~\ref{thm:converse.sos.switched.sys}). We also show that for switched linear systems (both in discrete and continuous
%time), if the inequality on the decrease condition of a Lyapunov function is satisfied as a sum of squares, then the Lyapunov
%function itself is automatically a sum of squares (Propositions~\ref{prop:switch.DT.V.automa.sos} and~\ref{prop:switch.CT.V.automa.sos}). 

\end{itemize}

Parts of this paper have previously appeared in several conference papers~\cite{AAA_PP_CDC11_converseSOS_Lyap}, \cite{AAA_MK_PP_CDC11_no_Poly_Lyap}, \cite{AAA_ACC12_Cubic_Difficult}, \cite{ACC13_complexity_CT10}, and in the PhD thesis of the first author~\cite[Chap. 4]{AAA_PhD}.

\section{Preliminaries} \label{sec:prelims}
We are concerned in this paper with a continuous time dynamical system
\begin{equation}\label{eq:CT.dynamics}
\dot{x}=f(x),
\end{equation}
where each \aaa{component} of $f:\mathbb{R}^n\rightarrow\mathbb{R}^n$ is a multivariate polynomial. We assume throughout and without loss of generality that $f$ has an equilibrium point at the origin, i.e., $f(0)=0$. We say that the origin is \emph{stable in the sense of Lyapunov} if $\forall \ \epsilon > 0$, $\exists \ \delta > 0$ such that $||x(0)|| < \delta \implies ||x(t)|| < \epsilon, \forall t > 0.$ The origin is said to be \emph{locally asymptotically stable} (LAS) if it is stable in the sense of Lyapunov and there exists $\epsilon > 0$ such that $||x(0)|| < \epsilon \implies \lim_{t \to \infty} x(t) = 0$. If the origin is stable in the sense of Lyapunov and $\forall x(0)\in\mathbb{R}^n$, we have $\lim_{t \to \infty} x(t) = 0$, then we say that the origin is \emph{globally asymptotically stable} (GAS). This is the notion of stability that we are mostly concerned with in this paper. We know from standard Lyapunov \aaa{theory} (see e.g.~\cite[Chap. 4]{Khalil:3rd.Ed}) that if we find a radially unbounded (Lyapunov) function $V(x):\mathbb{R}^n\rightarrow\mathbb{R}$ that vanishes at the origin and satisfies
\begin{eqnarray}
V(x)&>&0\quad \forall x\neq0 \label{eq:V.positive} \\
\dot{V}(x)=\langle\nabla V(x),f(x)\rangle&<&0\quad \forall x\neq0,
\label{eq:Vdot.negative}
\end{eqnarray}
then the origin of (\ref{eq:CT.dynamics}) is GAS. Here, $\dot{V}$ denotes the time derivative of $V$ along the
trajectories of (\ref{eq:CT.dynamics}), $\nabla V(x)$ is the
gradient vector of $V$, and $\langle .,. \rangle$ is the standard
inner product in $\mathbb{R}^n$. Similarly, existence of a Lyapunov function satisfying the two inequalities above on some ball around the origin would imply LAS. A real valued function $V$ satisfying the inequality in (\ref{eq:V.positive}) is said to be \emph{positive definite}. We say that $V$ is \emph{negative definite} if $-V$ is positive definite.

As discussed in the introduction, for stability analysis of polynomial systems it is most common (and quite natural)
to search for Lyapunov functions that are polynomials themselves. When such a candidate Lyapunov function is sought, the conditions in (\ref{eq:V.positive}) and (\ref{eq:Vdot.negative}) are two polynomial positivity conditions that $V$ needs to satisfy. The computational intractability of testing existence of such $V$ (when at least one of the two polynmials in (\ref{eq:V.positive}) or (\ref{eq:Vdot.negative}) have degree four or larger) leads us to the notion of sum of squares Lyapunov functions. Recall that a polynomial $p$ is a \emph{sum of squares} (sos), if it can be written as a sum of squares of polynomials, i.e., $p(x)=\sum_{i=1}^r q_i^2(x)$ for some integer $r$ and polynomials $q_i$.

\begin{definition}\label{def:sos.Lyap}
We say that a polynomial $V$ is a \emph{sum of squares Lyapunov function} for the polynomial vector field $f$ in (\ref{eq:CT.dynamics}) if the following two sos decomposition constraints are satisfied:
\begin{eqnarray}
V& \mbox{sos}\label{eq:V.SOS} \\
-\dot{V}=-\langle\nabla V,f\rangle& \mbox{sos}.
\label{eq:-Vdot.SOS}
\end{eqnarray}
\end{definition}

A sum of squares decomposition is a sufficient condition for polynomial nonnegativity\footnote{A polynomial $p$ is \emph{nonnegative} if $p(x)\geq 0$ for all $x\in\mathbb{R}^n$.} that can be checked with semidefinite programming. We do not present here the semidefinite program that decides if a given polynomial is sos since it has already appeared in several places. The unfamiliar reader is referred to~\cite{sdprelax},\aaa{~\cite{PabloGregRekha_BOOK}}. For a fixed degree of a polynomial Lyapunov candidate $V$, the search for the coefficients of $V$ subject to the constraints (\ref{eq:V.SOS}) and (\ref{eq:-Vdot.SOS}) is also a semidefinite program (SDP). The size of this SDP is polynomial in the size of the coefficients of the vector field $f$. Efficient algorithms for solving this SDP are available, for example those base on interior point methods~\cite{VaB:96}.

We emphasize that Definition~\ref{def:sos.Lyap} is the sensible definition of a sum of squares Lyapunov function and not what the name may suggest, which is a Lyapunov function that is a sum of squares. Indeed, the underlying semidefinite program will find a Lyapunov function $V$ if and only if $V$ satisfies \emph{both} conditions (\ref{eq:V.SOS}) and (\ref{eq:-Vdot.SOS}). We shall also remark that we think of the conditions in (\ref{eq:V.SOS}) and (\ref{eq:-Vdot.SOS}) as sufficient conditions
for the inequalities in (\ref{eq:V.positive}) and (\ref{eq:Vdot.negative}). Even though an sos polynomial can be merely nonnegative (as opposed to positive definite), when the underlying semidefinite programming feasibility problems are solved by interior point methods, solutions that are returned are generically positive definite; see the discussion in~\cite[p. 41]{AAA_MS_Thesis}. Later in the paper, when we want to prove results on existence of sos Lyapunov functions, we certainly require the proofs to imply existence of \emph{positive definite} sos Lyapunov functions.

Finally, we recall some basic facts about homogeneous vector fields. A real-valued function $p$ is said to be \emph{homogeneous} (of degree $d$) if it satisfies $p(\lambda x)=\lambda^d p(x)$ for any scalar $\lambda\in\mathbb{R}$. If $p$ is a polynomial, this condition is equivalent to all monomials of $p$ having the same degree $d$. It is easy to see that homogeneous polynomials are closed under sums and products and that the gradient of a homogeneous polynomial has entries that are homogeneous polynomials. A polynomial vector field $\dot{x}=f(x)$ is \emph{homogeneous} if all entries of $f$ are homogeneous polynomials of the same degree (see the vector field in (\ref{eq:poly8}) for an example). \aaa{Linear systems, for instance, are homogeneous polynomial vector fields of degree one. There is a large literature in nonlinear control theory on homogeneous vector fields~\cite{Stability_homog_poly_ODE}, \cite{Stabilize_Homog},
\cite{homog.feedback}, \cite{Baillieul_Homog_geometry},
\cite{Cubic_Homog_Planar}, \cite{HomogHomog},
\cite{homog.systems}}, and some of the results of this paper (both negative and positive) are derived specifically for this
class of systems. (Note that a negative result established for homogeneous polynomial vector fields is stronger than a negative result derived for general polynomial vector fields.) A basic fact about homogeneous systems is that for these vector fields the notions of local and global stability are equivalent. Indeed, a homogeneous vector field of degree $d$ satisfies $f(\lambda x)=\lambda^d f(x)$ for any scalar $\lambda$, and therefore the value of $f$ on the unit sphere determines its
value everywhere. It is also well-known that an asymptotically stable homogeneous system admits a homogeneous Lyapunov
funciton~\cite{Hahn_stability_book},\cite{HomogHomog}.

\section{Complexity of deciding stability for polynomial vector fields}\label{sec:complexity.cubic.vec.field}

The most basic algorithmic question \aaa{one} can ask about stability of equilibrium points of polynomial vector fields is whether \aaa{this} property can be decided in finite time. This is in fact a well-known question of Arnold that appears in~\cite{Arnold_Problems_for_Math}:

\vspace{-5pt}
\begin{itemize}
\item[] ``Is the stability problem for stationary points
algorithmically decidable? The well-known Lyapounov
theorem\footnote{The theorem that Arnold is referring to here is
the indirect method of Lyapunov related to linearization. This is
not to be confused with Lyapunov's direct method (or the second
method), which is what we are concerned with in sections that
follow.} solves the problem in the absence of eigenvalues with
zero real parts. In more complicated cases, where the stability
depends on higher order terms in the Taylor series, there exists
no algebraic criterion.

Let a vector field be given by polynomials of a fixed degree, with
rational coefficients. Does an algorithm exist, allowing to
decide, whether the stationary point is stable?''
\end{itemize}
\vspace{-5pt}

Later in~\cite{Costa_Doria_undecidabiliy}, the question of Arnold
is quoted with more detail:

\vspace{-5pt}
\begin{itemize}
\item[] ``In my problem the coefficients of the polynomials of
known degree and of a known number of variables are written on the
tape of the standard Turing machine in the standard order and in
the standard representation. The problem is whether there exists
an algorithm (an additional text for the machine independent of
the values of the coefficients) such that it solves the stability
problem for the stationary point at the origin (i.e., always stops
giving the answer ``stable'' or ``unstable'').

I hope, this algorithm exists if the degree is one.  It also
exists when the dimension is one. My conjecture has always been
that there is no algorithm for some sufficiently high degree and
dimension, perhaps for dimension $3$ and degree $3$ or even $2$. I
am less certain about what happens in dimension $2$. Of course the
nonexistence of a general algorithm for a fixed dimension working
for arbitrary degree or for a fixed degree working for an
arbitrary dimension, or working for all polynomials with arbitrary
degree and dimension would also be interesting.''
\end{itemize}
\vspace{-5pt}

To our knowledge, there has been no formal resolution to these
questions, neither for the case of stability in the sense of
Lyapunov, nor for the case of asymptotic stability (in its local
or global version). In~\cite{Costa_Doria_undecidabiliy}, da Costa
and Doria show that if the right hand side of the differential
equation contains elementary functions (sines, cosines,
exponentials, absolute value function, etc.), then there is no
algorithm for deciding whether the origin is stable or unstable.
They also present a dynamical system in~\cite{Costa_Doria_Hopf}
where one cannot decide whether a Hopf bifurcation will occur or
whether there will be parameter values such that a stable fixed
point becomes unstable. In earlier work, Arnold himself
demonstrates some of the difficulties that arise in stability
analysis of polynomial systems by presenting a parametric
polynomial system in $3$ variables and degree $5$, where the
boundary between stability and instability in parameter space is
not a semialgebraic set~\cite{Arnold_algebraic_unsolve}. \aaa{Similar approaches have been taken in the systems theory literature to show so-called ``algebraic unsolvability'' or ``rational undecidability'' of some fundamental stability questions in controls, such as that of simultaneous stabilizability of three linear systems~\cite{Blondel_Undecid_3Plants}, or stability of a pair of discrete time linear systems under arbitrary switching~\cite{JSR_not_Algebraic_Kozyakin}.} A relatively larger number of undecidability results are available
for questions related to other properties of polynomial vector
fields, such as
reachability~\cite{Undecidability_vec_fields_survey} or
boundedness of domain of
definition~\cite{Bounded_Defined_ODE_Undecidable}, or for
questions about stability of hybrid systems~\cite{TsiLinSat},
\cite{BlTi2}, \cite{BlTi_stab_contr_hybrid},
\cite{Deciding_stab_mortal_PWA}. We refer the interested reader to
the survey papers in~\cite{Survey_CT_Computation},
\cite{Undecidability_vec_fields_survey},
\cite{Sontag_complexity_comparison},
\cite{BlTi_complexity_3classes}, \cite{BlTi1}.

While we are also unable to give a proof of undecidability of testing stability, we can establish the following result which gives a lower bound on the complexity of the problem.

\begin{theorem}\label{thm:asym.stability.nphard}
Deciding (local or \aaa{global}) asymptotic stability of \aaa{cubic} polynomial
vector fields is strongly NP-hard. \aaa{This is true even when the vector field is restricted to be homogeneous.}
\end{theorem}

%The input to our stability problem is an ordered list of the (rational) coefficients of the polynomials appearing in the %vector field.

The implication of the NP-hardness of this problem is that unless
P=NP, it is impossible to design an algorithm that can take as
input the (rational) coefficients of a homogeneous cubic vector
field, have running time bounded by a polynomial in the number of
bits needed to represent these coefficients, and always output the
correct yes/no answer on asymptotic stability. Moreover, the fact
that our NP-hardness result is in the strong sense (as opposed to
weakly NP-hard problems such as KNAPSACK, SUBSET SUM, etc.)
implies that the problem remains NP-hard even if the size (bit
length) of the coefficients is $O(\log n)$, where $n$ is the
dimension. For a strongly NP-hard problem, even a
pseudo-polynomial time algorithm cannot exist unless P=NP.
See~\cite{GareyJohnson_Book} for precise definitions and more
details. It would be very interesting to settle whether for the restricted class of homogeneous vector fields, the stability testing problem is decidable.

Theorem~\ref{thm:asym.stability.nphard} in particular suggests that unless P=NP, we
should not expect sum of squares Lyapunov functions of ``low
enough'' degree to always exist, even when the analysis is
restricted to cubic homogeneous vector fields. The semidefinite
program arising from a search for an sos Lyapunov function of
degree $2d$ for such a vector field in $n$ variables has size in
the order of ${n+d \choose d+1}$. This number is polynomial in $n$
for fixed $d$ (but exponential in $n$ when $d$ grows linearly in
$n$). Therefore, unlike the case of linear systems for which sos quadratic Lyapunov functions always exist, we should not
hope to have a bound on the degree of sos Lyapunov functions that
is independent of the dimension. We postpone our study of existence of sos Lyapunov functions to
Section~\ref{sec:(non)-existence.sos.lyap} and continue for now with our complexity explorations.

Before the proof of Theorem~\ref{thm:asym.stability.nphard} is presented, let us remark that this hardness result is minimal in the degree as far as homogeneous vector fields are concerned. It is not hard to show that for quadratic homogeneous vector fields (and in fact for all even degree homogeneous vector fields), the origin can never be asymptotically stable; see e.g.~\cite[p.
283]{Hahn_stability_book}. This leaves us with homogeneous polynomial vector fields of degree one (i.e., linear systems) for which asymptotic stability can be decided in polynomial time, e.g. as shown in the following \aaa{folklore} lemma.

\begin{lemma} \label{lem:linear.in.P}
There is a polynomial time algorithm for deciding asymptotic stability of the equilibrium point of the linear system $\dot{x}=Ax$.\footnote{This of course is equivalent to the matrix $A$ being Hurwitz, i.e., having all eigenvalues in the open left half complex plane. Theorem~\ref{thm:asym.stability.nphard} suggests that no simple characterization of this type should be possible for cubic vector fields.}
\end{lemma}

\begin{proof}
Consider the following algorithm. Given the matrix $A$, solve the following linear system \aaa{(the Lyapunov equation)} for the symmetric matrix $P$: $$A^TP+PA=-I,$$ where $I$ here is the identity matrix of the same dimension as $A$. If there is no solution to this system, output ``no''. If a solution matrix $P$ is found, test whether $P$ is a positive definite matrix. If yes, output ``yes''. If not, output ``no''. The fact that this algorithm is correct is standard. Moreover, the algorithm runs in polynomial time (in fact in $O(n^3)$) since linear systems can be solved in polynomial time, and we can decide if a matrix is positive definite e.g. by checking whether all its $n$ leading \aaa{principal} minors are positive rational numbers. Each determinant can be computed in polynomial time.
\end{proof}

\begin{corollary}\label{cor:local.exp.easy}
\aaa{L}ocal exponential stability\footnote{See~\cite{Khalil:3rd.Ed} for a definition.} of polynomial vector fields of any degree can be decided in polynomial time.
\end{corollary}

\begin{proof}
A polynomial vector field is locally exponentially stable if and only if its linearization is Hurwitz~\cite{Khalil:3rd.Ed}. The linearization matrix can trivially be written down in polynomial time and its Hurwitzness can be checked by the algorithm presented in the previous lemma.
\end{proof}

Note that Theorem~\ref{thm:asym.stability.nphard} and Corollary~\ref{cor:local.exp.easy} draw a sharp contrast between the complexity of checking local asymptotic stability versus local exponential stability. Observe also that homogeneous vector fields of degree larger than one can never be exponentially stable (since their linearization matrix is the zero matrix). \aaa{Indeed, the difficulty in deciding asymptotic stability initiates when the linearization of the vector field has eigenvalues on the imaginary axis, so that Lyapunov's linearization test cannot conclusively be employed to test stability.}

We now proceed to the proof of Theorem~\ref{thm:asym.stability.nphard}. The main intuition behind this proof is the
following idea: We will relate the solution of a combinatorial
problem not to the behavior of the trajectories of a cubic vector
field that are hard to get a handle on, but instead to properties
of a Lyapunov function that proves asymptotic stability of this
vector field. As we will see shortly, insights from Lyapunov
theory make the proof of this theorem \aaa{relatively} simple. The reduction
is broken into two steps:
\begin{center}
ONE-IN-THREE 3SAT \\ $\downarrow$ \\ positivity of quartic forms
\\ $\downarrow$\\  asymptotic stability of cubic vector fields
\end{center}
%In the course of presenting these reductions, we will also discuss
%some corollaries that are not directly related to our study of
%asymptotic stability, but are of independent interest.

%First, we give a reduction from ONE-IN-THREE 3SAT to the problem
%of deciding positive definiteness of quartic forms. Then, we give
%a reduction from the latter problem to the problem of deciding
%asymptotic stability of cubic vector fields.

\subsection{Reduction from ONE-IN-THREE 3SAT to positivity of quartic forms}\label{subsec:3SAT_2_positivity}

It is well-known that deciding nonnegativity (i.e., positive semidefiniteness) of
quartic forms is NP-hard. The proof commonly cited in the
literature is based on a reduction from the matrix copositivity
problem~\cite{nonnegativity_NP_hard}: given a symmetric $n \times
n$ matrix $Q$, decide whether $x^TQx\geq0$ for all $x$'s that are
elementwise nonnegative. Clearly, a matrix $Q$ is copositive if
and only if the quartic form $z^TQz$, with
$z_i\mathrel{\mathop:}=x_i^2$, is nonnegative. The original
reduction~\cite{nonnegativity_NP_hard} proving NP-hardness of
testing matrix copositivity is from the subset sum problem and
only establishes weak NP-hardness. However, reductions from the
stable set problem to matrix copositivity are also
known~\cite{deKlerk_StableSet_copositive},
\cite{copositivity_NPhard} and they result in NP-hardness in the
strong sense.

For reasons that will become clear shortly, we are interested in
showing hardness of deciding \emph{positive definiteness} of
quartic forms as opposed to positive semidefiniteness. This is in
some sense even easier to accomplish. A very straightforward
reduction from 3SAT proves NP-hardness of deciding positive
definiteness of polynomials of degree $6$. By using ONE-IN-THREE
3SAT instead, we will reduce the degree of the polynomial from $6$
to $4$.

\begin{proposition}\label{prop:positivity.quartic.NPhard}
It is strongly\footnote{The NP-hardness results of
this section will all be in the strong sense. From here on, we
will drop the prefix ``strong'' for brevity.} NP-hard to decide
whether a homogeneous polynomial of degree~$4$ is positive
definite.
\end{proposition}

\begin{proof}
We give a reduction from ONE-IN-THREE 3SAT which is known to be
NP-complete~\cite[p. 259]{GareyJohnson_Book}. Recall that in
ONE-IN-THREE 3SAT, we are given a 3SAT instance (i.e., a
collection of clauses, where each clause consists of exactly three
literals, and each literal is either a variable or its negation)
and we are asked to decide whether there exists a $\{0,1\}$
assignment to the variables that makes the expression true with
the additional property that each clause has \emph{exactly one}
true literal.

To avoid introducing unnecessary notation, we present the
reduction on a specific instance. The pattern will make it obvious
that the general construction is no different. Given an instance
of ONE-IN-THREE 3SAT, such as the following
\begin{equation}\label{eq:reduciton.1-in-3.3sat.instance}
(x_1\vee\bar{x}_2\vee x_4)\wedge (\bar{x}_2\vee\bar{x}_3\vee
x_5)\wedge (\bar{x}_1\vee x_3\vee \bar{x}_5)\wedge (x_1\vee
x_3\vee x_4),
\end{equation}
we define the quartic polynomial $p$ as follows:
\begin{equation}\label{eq:reduction.p}
\begin{array}{lll}
p(x)&=&\sum_{i=1}^5 x_i^2(1-x_i)^2\\ \ &\
&+(x_1+(1-x_2)+x_4-1)^2+((1-x_2)+(1-x_3)+x_5-1)^2
\\ \ &\ &+((1-x_1)+x_3+(1-x_5)-1)^2+(x_1+x_3+x_4-1)^2.
\end{array}
\end{equation}
Having done so, our claim is that $p(x)>0$ for all $x\in
\mathbb{R}^5$ (or generally for all $x\in \mathbb{R}^n$) if and
only if the ONE-IN-THREE 3SAT instance is not satisfiable. Note
that $p$ is a sum of squares and therefore nonnegative. The only
possible locations for zeros of $p$ are by construction among the
points in $\{0,1\}^5$. If there is a satisfying Boolean assignment
$x$ to (\ref{eq:reduciton.1-in-3.3sat.instance}) with exactly one
true literal per clause, then $p$ will vanish at point $x$.
Conversely, if there are no such satisfying assignments, then for
any point in $\{0,1\}^5$, at least one of the terms in
(\ref{eq:reduction.p}) will be positive and hence $p$ will have no
zeros.

It remains to make $p$ homogeneous. This can be done via
introducing a new scalar variable $y$. If we let
\begin{equation}\label{eq:reduction.ph}
p_h(x,y)=y^4 p(\textstyle{\frac{x}{y}}),
\end{equation}
then we claim that $p_h$ (which is a quartic form) is positive
definite if and only if $p$ constructed as in
(\ref{eq:reduction.p}) has no zeros.\footnote{In general,
homogenization does not preserve positivity. For example, as shown
in~\cite{Reznick}, the polynomial $x_1^2+(1-x_1x_2)^2$ has no
zeros, but its homogenization $x_1^2y^2+(y^2-x_1x_2)^2$ has zeros
at the points $(1,0,0)^T$ and $(0,1,0)^T$. Nevertheless,
positivity is preserved under homogenization for the special class
of polynomials constructed in this reduction, essentially because
polynomials of type (\ref{eq:reduction.p}) have no zeros at
infinity.} Indeed, if $p$ has a zero at a point $x$, then that
zero is inherited by $p_h$ at the point $(x,1)$. If $p$ has no
zeros, then (\ref{eq:reduction.ph}) shows that $p_h$ can only
possibly have zeros at points with $y=0$. However, from the
structure of $p$ in (\ref{eq:reduction.p}) we see that
$$p_h(x,0)=x_1^4+\cdots+x_5^4,$$ which cannot be zero (except at
the origin). This concludes the proof.
\end{proof}

\subsection{Reduction from positivity of quartic forms to asymptotic stability of cubic vector fields}
We now present the second step of the reduction and finish the
proof of Theorem~\ref{thm:asym.stability.nphard}.

\begin{proof}[Proof of Theorem~\ref{thm:asym.stability.nphard}]
We give a reduction from the problem of deciding positive
definiteness of quartic forms, whose NP-hardness was established
in Proposition~\ref{prop:positivity.quartic.NPhard}. Given a
quartic form $V\mathrel{\mathop:}=V(x)$, we define the polynomial
vector field
\begin{equation}\label{eq:xdot.cubic.reduction}
\dot{x}=-\nabla V(x).
\end{equation}
Note that the vector field is homogeneous of degree $3$. We claim
that the above vector field is (locally or equivalently globally)
asymptotically stable if and only if $V$ is positive definite.
First, we observe that by construction
\begin{equation}\label{eq:Vdot<=0.always}
\dot{V}(x)=\langle \nabla V(x), \dot{x} \rangle=-||\nabla
V(x)||^2\leq 0.
\end{equation}
Suppose $V$ is positive definite. By Euler's identity for
homogeneous functions,\footnote{Euler's identity\aaa{, $V(x)=\frac{1}{d}x^T\nabla V(x)$,} is easily derived
by differentiating both sides of the equation $V(\lambda
x)~=~\lambda^d V(x)$ with respect to $\lambda$ and setting
$\lambda=1$.} we have $V(x)=\frac{1}{4}x^T\nabla V(x).$ Therefore,
positive definiteness of $V$ implies that $\nabla V(x)$ cannot
vanish anywhere except at the origin. Hence, $\dot{V}(x)<0$ for
all $x\neq 0$. In view of Lyapunov's theorem (see e.g.~\cite[p.
124]{Khalil:3rd.Ed}), and the fact that a positive definite
homogeneous function is radially unbounded, it follows that the
system in (\ref{eq:xdot.cubic.reduction}) is globally
asymptotically stable.

For the converse direction, suppose
(\ref{eq:xdot.cubic.reduction}) is GAS. Our first claim is that
global asymptotic stability together with $\dot{V}(x)\leq 0$
implies that $V$ must be positive semidefinite. This follows from
the following simple argument, which we have also previously
presented in~\cite{AAA_PP_ACC11_Lyap_High_Deriv} for a different
purpose. Suppose for the sake of contradiction that for some
$\hat{x}\in\mathbb{R}^n$ and some $\epsilon>0,$ we had
$V(\hat{x})=-\epsilon<0$. Consider a trajectory $x(t;\hat{x})$ of
system (\ref{eq:xdot.cubic.reduction}) that starts at initial
condition $\hat{x}$, and let us evaluate the function $V$ on this
trajectory. Since $V(\hat{x})=-\epsilon$ and $\dot{V}(x)\leq 0$,
we have $V(x(t;\hat{x}))\leq-\epsilon$ for all $t>0$. However,
this contradicts the fact that by global asymptotic stability, the
trajectory must go to the origin, where $V$, being a form,
vanishes.

To prove that $V$ is positive definite, suppose by contradiction
that for some nonzero point $x^*\in\mathbb{R}^n$ we had
$V(x^*)=0$. Since we just proved that $V$ has to be positive
semidefinite, the point $x^*$ must be a global minimum of $V$.
Therefore, as a necessary condition of optimality, we should have
$\nabla V(x^*)=0$. But this contradicts the system in
(\ref{eq:xdot.cubic.reduction}) being GAS, since the trajectory
starting at $x^*$ stays there forever and can never go to the
\aaa{origin}.\footnote{After a presentation of this work at UCLouvain, P.-A. Absil kindly brought to our attention that similar connections between local optima of real analytic functions and local stability of their gradient systems appear in~\cite{Absil_local_min_local_stability}.}
\end{proof}

\subsection{Complexity of deciding other qualitative properties of the trajectories}

\aaa{We next} establish NP-hardness of several other decision questions associated with polynomial vector fields, which arise in numerous areas of systems and control theory. \aaa{Most of} the results will be rather straightforward corollaries of Theorem~\ref{thm:asym.stability.nphard}. In what follows, whenever a property has to do with an equilibrium point, we take this equilibrium point to be at the origin. The norm $||.||$ is always the Euclidean norm, and the notation $B_r$ denotes the ball of radius $r$; i.e., $B_r\mathrel{\mathop:}=\{x\ |\ ||x||\leq r\}$.

\begin{theorem} [see also~\cite{ACC13_complexity_CT10}]\label{thm:poly.hardness.results}
For polynomial differential equations of degree $d$ (with $d$ specified below), the following properties are NP-hard to decide:\\
%(a) $d=3$, \emph{Invariance of a ball}: $\forall x(0)$ with $||x(0)||\leq 1$, $$||x(t)||\leq 1, \ \ \forall t\geq0.$$\\
%(b) $d=1$, \emph{Invariance of a basic semialgebraic set defined by a quartic polynomial}: $\forall x(0)\in\mathcal{S}$, 
%$$x(t)\in\mathcal{S}, \ \ \forall t\geq0,$$ where $\mathcal{S}\mathrel{\mathop:}=\{x\ |\ p(x)\leq 1\}$ and $p$ is a given form %of degree four.\\
(a) $d=3$, \emph{Inclusion of a ball in the region of attraction of an equilibrium point}: $\forall x(0)$ with $||x(0)||\leq 1,$  $$x(t)\rightarrow 0, \ \mbox{as}\ \  t\rightarrow\infty.$$ \\
(b) $d=3$, \emph{Local attractivity of an equilibrium point}: $\exists \delta>0$ such that $\forall x(0)\in B_\delta,$ $$x(t)\rightarrow 0, \ \mbox{as}\ \  t\rightarrow\infty.$$ \\
(c) $d=4$, \emph{Stability of an \aaa{equilibrium} point in the sense of Lyapunov}: $\forall \epsilon>0$, $\exists\delta=~\delta(\epsilon)$ such that $$||x(0)||<\delta\ \Rightarrow ||x(t)||<\epsilon, \ \ \forall t\geq0.$$ \\
(d) $d=3$, \emph{Boundedness of trajectories}: $\forall x(0)$, $\exists r=r(x(0))$ such that
$$||x(t)||<r, \ \ \forall t\geq0.$$ \\
(e) $d=3$, \emph{Existence of a local quadratic Lyapunov function}: $\exists\delta>0$ and a quadratic Lyapunov function $V(x)=x^TPx$ such that $V(x)>0$ for all $x\in B_\delta, \ x\neq0$ (or equivalently $P\succ 0$), and $$\dot{V}(x)<0, \ \ \forall x\in B_\delta, \ x\neq0.$$\\
(f) $d=4$, \emph{Local collision avoidance}: $\exists\delta>0$ such that $\forall x(0)\in B_\delta$, $$x(t)\notin\mathcal{S}, \  \ \forall t\geq0,$$
where $\mathcal{S}$ is a given polytope.\\
(g) $d=3$, \emph{Existence of a stabilizing controller}: There exists a particular (e.g. smooth, or polynomial of fixed degree) control law $u(x)$ that makes the origin of $$\dot{x}=f(x)+g(x)u(x)$$ locally asymptotically stable, where $f$ and $g\neq 0$ here have degrees $3$.\footnote{The size of $g$ and $u$ here are respectively $n\times k$ and $k\times 1$ and the statement is true for any $k\in\{1,\ldots,n\}$. We present the proof for $k=n$; the other proofs are identical.}\\
\aaa{(h) $d=3$, \emph{Invariance of a ball}: $\forall x(0)$ with $||x(0)||\leq 1$, $$||x(t)||\leq 1, \ \ \forall t\geq0.$$\\
(i) $d=1$, \emph{Invariance of a basic semialgebraic set defined by a quartic polynomial}: $\forall x(0)\in\mathcal{S}$, 
$$x(t)\in\mathcal{S}, \ \ \forall t\geq0,$$ where $\mathcal{S}\mathrel{\mathop:}=\{x\ |\ p(x)\leq 1\}$ and $p$ is a given form of degree four.}\\

\end{theorem}

\begin{proof}

The proofs \aaa{of (a)-(g)} will be via reductions from the problem of testing local asymptotic stability (LAS) of cubic vector fields established in Theorem~\ref{thm:asym.stability.nphard}. These reductions in fact depend on the specific structure of the vector field constructed in the proof of Theorem~\ref{thm:asym.stability.nphard}:
\begin{equation}\label{eq:xdot=f(x)=-gradV(x)}
\dot{x}\mathrel{\mathop:}=f(x)=-\nabla V(x),
\end{equation}
where $V$ is a quartic form. We recall some facts:
\begin{enumerate}
\item The vector field in (\ref{eq:xdot=f(x)=-gradV(x)}) is homogeneous and therefore it is locally asymptotically stable if and only if it is globally asymptotically stable.
\item The vector field is (locally or globally) asymptotically stable if and only if $V$ is positive definite. Hence, the system, if asymptotically stable, by construction always admits a quartic Lyapunov function.
\item If the vector field is \emph{not} asymptotically stable, then there always exists a nonzero point $\bar{x}\in\{0,1\}^n$ such that $f(\bar{x})=0$; i.e., $\bar{x}$ is a nonzero equilibrium point.
\end{enumerate}
We now proceed with the proofs. In what follows, $f(x)$ will always refer to the vector field in (\ref{eq:xdot=f(x)=-gradV(x)}).

(a) The claim is an obvious implication of the homogeneity of $f$. Since $f(\lambda x)=\lambda^3 f(x)$, for all $\lambda\in\mathbb{R}$ and all $x\in\mathbb{R}^n$, the origin is LAS if and only if for any $r$, all trajectories in $B_r$ converge to the origin.\footnote{For a general cubic vector field, validity of property \aaa{(a)} for a particular value of $r$ is of course not necessary for local asymptotic stability. The reader should keep in mind that the class of homogeneous cubic vector fields is a subset of the class of all cubic vector fields, and hence any hardness result for this class immediately implies the same hardness result for all cubic vector fields.}

(b) If $f$ is LAS, then of course it is by definition locally attractive. On the other hand, if $f$ is not LAS, then $f(\bar{x})=0$ for some nonzero $\bar{x}\in\{0,1\}^n$. By homogeneity of $f$, this implies that $f(\alpha\bar{x})=0, \ \forall\alpha\geq 0.$ Therefore, arbitrarily close to the origin we have stationary points and hence the origin cannot be locally attractive.

(c) Let $x^4\mathrel{\mathop:}=(x_1^4,\ldots, x_n^4)^T.$ Consider the vector field
\begin{equation}\label{eq:xdot=f(x)+x^4}
\dot{x}=f(x)+x^4.
\end{equation}
We claim that the origin of (\ref{eq:xdot=f(x)+x^4}) is stable in the sense of Lyapunov if and only if the origin of (\ref{eq:xdot=f(x)=-gradV(x)}) is LAS. Suppose first that (\ref{eq:xdot=f(x)=-gradV(x)}) is not LAS. Then we must have $f(\alpha\bar{x})=0$ for some nonzero $\bar{x}\in\{0,1\}^n$ and $\forall\alpha\geq 0$. Therefore for the system (\ref{eq:xdot=f(x)+x^4}), trajectories starting from \emph{any} nonzero point on the line connecting the origin to $\bar{x}$ shoot out to infinity while staying on the line. (This is because on this line, the dynamics are simply $\dot{x}=x^4.$) As a result, stability in the sense of Lyapunov does not hold as there exists an $\epsilon>0$ (in fact for any $\epsilon>0$), for which $\nexists\delta>0$ such that trajectories starting in $B_\delta$ stay in $B_\epsilon$. Indeed, as we argued, arbitrarily close to the origin we have points that shoot out to infinity. 

Let us now show the converse. If (\ref{eq:xdot=f(x)=-gradV(x)}) is LAS, then $V$ is indeed a strict Lyapunov function for it; i.e. it is positive definite and has a negative definite derivative $-||\nabla V(x)||^2$. Using the same Lyapunov function for the system in (\ref{eq:xdot=f(x)+x^4}), we have $$\dot{V}(x)=-||\nabla V(x)||^2+\langle \nabla V(x), x^4\rangle.$$ Note that the first term in this expression is a homogeneous polynomial of degree $6$ while the second term is a homogeneous polynomial of degree $7$. Negative definiteness of the lower order term implies that there exists a positive real number $\delta$ such that $\dot{V}(x)<0$ for all nonzero $x\in B_\delta$. This together with positive definiteness of $V$ implies via Lyapunov's theorem that (\ref{eq:xdot=f(x)+x^4}) is LAS and hence stable in the sense of Lypunov.

(d) Consider the vector field
\begin{equation}\label{eq:xdot=f(x)+x}
\dot{x}=f(x)+x.
\end{equation}
We claim that the trajectories of (\ref{eq:xdot=f(x)+x}) are bounded if and only if the origin of (\ref{eq:xdot=f(x)=-gradV(x)}) is LAS. Suppose (\ref{eq:xdot=f(x)=-gradV(x)}) is not LAS. Then, as in the previous proof, there exists a line connecting the origin to a point $\bar{x}\in\{0,1\}^n$ such that trajectories on this line escape to infinity. (In this case, the dynamics on this line is governed by $\dot{x}=x$.) Hence, not all trajectories can be bounded. For the converse, suppose that (\ref{eq:xdot=f(x)=-gradV(x)}) is LAS. Then $V$ (resp. $-||\nabla V(x)||^2$) must be positive (resp. negative) definite. Now if we consider system (\ref{eq:xdot=f(x)+x}), the derivative of $V$ along its trajectories is given by $$\dot{V}(x)=-||\nabla V(x)||^2+\langle \nabla V(x), x\rangle.$$ Since the first term in this expression has degree $6$ and the second term degree $4$, there exists an $r$ such that $\dot{V}<0$ for all $x\notin B_r$. This condition however implies boundedness of trajectories; see e.g.~\cite{Higher_Derive_Lagrange_Stability}.

(e) If $f$ is not LAS, then there cannot be any local Lyapunov functions, in particular not a quadratic one. If $f$ is LAS, then we claim the quadratic function $W(x)=||x||^2$ is a valid (and in fact global) Lyapunov function for it. This can be seen from Euler's identity $$\dot{W}(x)=\langle 2x, -\nabla V(x)  \rangle=-8V(x),$$
and by noting that $V$ must be positive definite.

(f) We define our dynamics to be the one in (\ref{eq:xdot=f(x)+x^4}), and the polytope $\mathcal{S}$ to be $$\mathcal{S}=\{x\ |\ x_i\geq 0, 1\leq \sum_{i=1}^n x_i \leq 2\}.$$ 

Suppose first that $f$ is not LAS. Then by the argument given in (e), the system in (\ref{eq:xdot=f(x)+x^4}) has trajectories that start out arbitrarily close to the origin (at points of the type $\alpha\bar{x}$ for some $\bar{x}\in\{0,1\}^n$ and for arbitrarily small $\alpha$), which exit on a straight line to infinity. Note that by doing so, such trajectories must cross $\mathcal{S}$; i.e. there exists a positive real number $\bar{\alpha}$ such that $1\leq \sum_{i=1}^n \bar{\alpha}\bar{x}_i \leq 2.$ Hence, there is no neighborhood around the origin whose trajectories avoid $\mathcal{S}$. 

For the converse, suppose $f$ is LAS. Then, we have shown while proving (e) that (\ref{eq:xdot=f(x)+x^4}) must also be LAS and hence stable in the sense of Lyapunov. Therefore, there exists $\delta>0$ such that trajectories starting from $B_\delta$ do not leave $B_{\frac{1}{2}}$---a ball that is disjoint from $\mathcal{S}$.

(g) Let $f$ be as in (\ref{eq:xdot=f(x)=-gradV(x)}) and $g(x)=(x_1x_2^2-x_1^2x_2)\boldsymbol{1}\boldsymbol{1} ^T$, where $\boldsymbol{1}$ denotes the vector of all ones. The following simple argument establishes the desired NP-hardness result irrespective of the type of control law we may seek (e.g. linear control law, cubic control law, smooth control law, or anything else). If $f$ is LAS, then of course there exists a stabilizing controller, namely $u=0$. If $f$ is not LAS, then it must have an equilibrium point at a nonzero point $\bar{x}\in\{0,1\}^n$. Note that by construction, $g$ vanishes at all such points. Since $g$ is homogeneous, it also vanishes at all points $\alpha\bar{x}$ for any scalar $\alpha$. Therefore, arbitrarily close to the origin, there are equilibrium points that the control law $u(x)$ cannot possibly remove. Hence there is no controller that can make the origin LAS.

\aaa{

The proofs of parts (h) and (i) of the theorem are based on a reduction from the polynomial nonnegativity problem for quartics: given a (homogeneous) degree-4 polynomial $p:\mathbb{R}^n\rightarrow\mathbb{R}$, decide whether $p(x)\geq0, \forall x\in\mathbb{R}^n$? (See the opening paragraph of Subsection~\ref{subsec:3SAT_2_positivity} for references on NP-hardness of this problem.)

(h) Given a quartic form $p$, we construct the vector field $$\dot{x}=-\nabla p(x).$$ Note that the vector field has degree $3$ and is homogeneous. We claim that the unit ball $B_1$ is invariant under the trajectories of this system if and only if $p$ is nonnegative. This of course establishes the desired NP-hardness result. To prove the claim, consider the function $V(x)\mathrel{\mathop:}=||x||^2.$ Clearly, $B_1$ is invariant under the trajectories of $\dot{x}=-\nabla p(x)$ if and only if $\dot{V}(x)\leq 0$ for all $x$ with $||x||=1.$ Since $\dot{V}$ is a homogeneous polynomial, this condition is equivalent to having $\dot{V}$ nonpositive for all $x\in\mathbb{R}^n$. However, from Euler's identity we have $$\dot{V}(x)=\langle 2x, -\nabla p(x) \rangle=-8p(x).$$

(i) Once again, we provide a reduction from the problem of checking nonnegativity of quartic forms. Given a quartic form $p$, we let the set $\mathcal{S}$ be defined as $\mathcal{S}=\{x\ |\ p(x)\leq 1\}$. Let us consider the linear dynamical system $$\dot{x}=-x.$$ We claim that $\mathcal{S}$ is invariant under the trajectories of this linear system if and only if $p$ is nonnegative. To see this, consider the derivative $\dot{p}$ of $p$ along the trajectories of $\dot{x}~=~-~x$ and note its homogeneity. With the same reasoning as in the proof of part (h), $\dot{p}(x)\leq 0$ for all $x\in\mathbb{R}^n$ if and only if the set $\mathcal{S}$ is invariant. From Euler's identity, we have
$$\dot{p}(x)=\langle \nabla p(x), -x \rangle=-4p(x).$$ Note that the role of the dynamics and the gradient of the ``Lyapunov function'' are swapped in the proofs of (h) and (i). 

}

\end{proof}

\begin{remark}
Arguments similar to the one presented in the proof of (g) above can be given to show NP-hardness of deciding existence of a controller that establishes several other properties, e.g., boundedness of trajectories, inclusion of the unit ball in the region of attraction, etc. In the statement of (f), the fact that the set $\mathcal{S}$ is a polytope is clearly arbitrary. This choice is only made because ``obstacles'' are most commonly modeled in the literature as polytopes. We also note that a related problem of interest here is that of deciding, given two polytopes, whether all trajectories starting in one avoid the other. This question is the complement of the usual reachability question, for which claims of undecidability have already appeared~\cite{Undecidability_vec_fields_survey}; see also~\cite{Skolem_Pisot_problem_CT}.
\end{remark}

\begin{remark}
\aaa{In~\cite{ACC13_complexity_CT10}, NP-hardness of testing local asymptotic stability is established also for vector fields that are \emph{trigonometric}. Such vector fields appear commonly in the field of robotics.}
\end{remark}

%\begin{remark}
%In~\cite{ACC13_complexity_CT10}, NP-hardness of three additional properties have been established via reductions of similar flavor: (i) testing invariance of the unit ball for cubic vector fields, (ii) testing invariance of a quartic semialgebraic set for linear vector fields, (iii) testing local asymptotic stability of \emph{trigonometric} vector fields of degree four.
%\end{remark}

\section{Non-existence of polynomial Lyapunov functions or degree bounds}\label{sec:no.poly.Lyap}

\subsection{Non-existence of polynomial Lyapunov functions}
As mentioned in the introduction, the question of
global asymptotic stability of polynomial vector fields is
commonly addressed by seeking a Lyapunov function that is
polynomial itself. This approach has become further prevalent over
the past decade due to the fact that we can use sum of squares
techniques to algorithmically search for such Lyapunov functions.
The question therefore naturally arises as to whether existence of
polynomial Lyapunov functions is necessary for global stability of
polynomial systems. \aaa{This question appears explicitly, e.g., in~\cite[Sect. VIII]{Peet.Antonis.converse.sos.CDC}.} In this section, we \aaa{present} a remarkably simple counterexample which shows that the answer to this question is negative. In view of the fact that globally asymptotically stable linear
systems always admit quadratic Lyapunov functions, it is quite
interesting to observe that the following vector field that is
arguably ``the next simplest system'' to consider does not admit a
polynomial Lyapunov function of any degree.

\begin{theorem} [Ahmadi, Krstic, Parrilo,~\cite{AAA_MK_PP_CDC11_no_Poly_Lyap}] \label{thm:GAS.poly.no.poly.Lyap}
Consider the polynomial vector field
\begin{equation}\label{eq:counterexample.sys}
\begin{array}{rll}
\dot{x}&=&-x+xy \\
\dot{y}&=&-y.
\end{array}
\end{equation}
The origin is a globally asymptotically stable equilibrium point,
but the system does not admit a polynomial Lyapunov function.
\end{theorem}

The proof of this theorem is presented in~\cite{AAA_MK_PP_CDC11_no_Poly_Lyap} and omitted from here. \aaa{Global asymptotic stability is established by} presenting a Lyapunov function that involves the logarithm function. Non-existence of polynomial Lyapunov functions has to do with ``fast growth rates'' far away from the origin; see some typical trajectories of this vector field in Figure~\ref{fig:trajectories} and \aaa{reference}~\cite{AAA_MK_PP_CDC11_no_Poly_Lyap} for more details.

\begin{figure}%[thpb]
\centering
\includegraphics[width=.4\columnwidth]{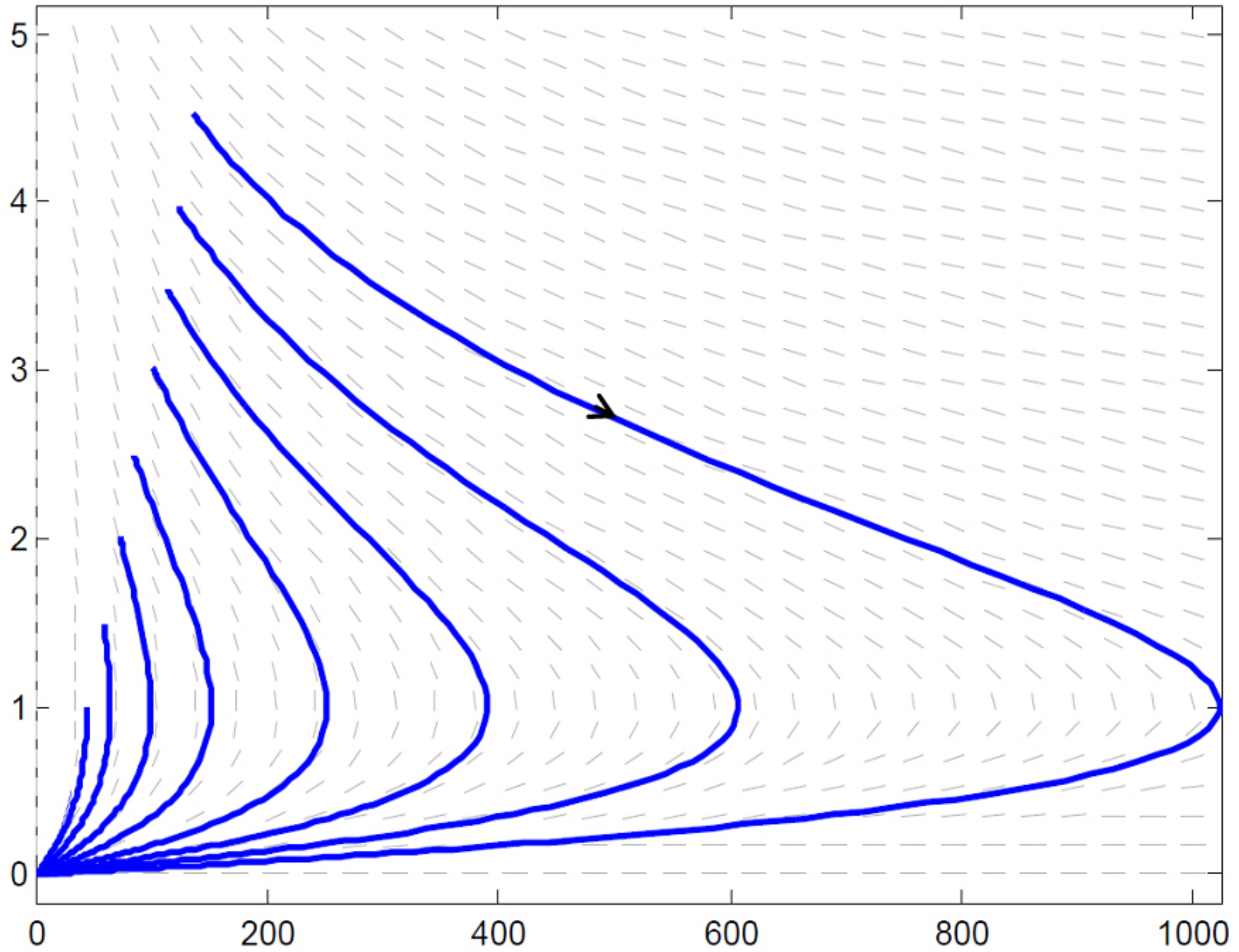}
\caption{Typical trajectories of the vector field in (\ref{eq:counterexample.sys}) starting from initial conditions in the nonnegative orthant.} \label{fig:trajectories}
\end{figure}

\paragraph{Example of Bacciotti and Rosier.} \aaa{An independent counterexample to existence of polynomial Lyapunov functions appears in a book by Bacciotti and Rosier~\cite[Prop.
5.2]{Bacciotti.Rosier.Liapunov.Book}. The differences between the two counterexamples are explained in some detail in~\cite{AAA_MK_PP_CDC11_no_Poly_Lyap}, the main one being that the example of Bacciotti and Rosier relies crucially on a coefficient appearing in the vector field being an \emph{irrational} number and is not robust to perturbations. (In practical applications where
computational techniques for searching over Lyapunov functions on
finite precision machines are used, such issues with irrationality
of the input cannot occur.) On the other hand, the example of Bacciotti and Rosier rules out existence of a polynomial (or even analytic) Lyapunov function \emph{locally}, as the argument is based on slow decay rates arbitrarily close to the origin. In~\cite{Peet.exp.stability}, Peet shows that locally
exponentially stable polynomial vector fields admit polynomial
Lyapunov functions on compact sets. The example of Bacciotti and
Rosier implies that the assumption of exponential stability indeed
cannot be dropped.}

\subsection{Homogeneous vector fields: non-existence of a uniform bound on the degree of polynomial Lyapunov functions in fixed dimension and
degree}\label{sec:no.finite.bound} %change this label later to subsection. be careful with references to it.

In this subsection, we restrict attention to polynomial vector fields that are homogeneous. Although the results of the last subsection imply that existence of polynomial Lyapunov functions is not necessary for global asymptotic stability of polynomial vector fields, the situation seems to be different for homogeneous vector fields. Although we have not been able to formally prove this, we conjecture the following:

\begin{itemize}
\item[]
Existence of a \emph{homogeneous polynomial Lyapunov function} is necessary and sufficient for global (or equivalently local) asymptotic stability of a \emph{homogeneous polynomial vector field}.
\end{itemize}
\vspace{-5pt}

One reason why we are interested in proving this conjecture is that in Section~\ref{sec:(non)-existence.sos.lyap}, we will show that for homogeneous polynomial vector fields, existence of a homogeneous polynomial Lyapunov function implies existence of a homogeneous \emph{sum of squares} Lyapunov function. Hence, the two statements put together would imply that stability of homogeneous vector fields can always be checked via sum of squares techniques (and hence semidefinite programming).

%For polynomial vector fields in general, existence of a polynomial
%Lyapunov function is not necessary for global asymptotic
%stability. In joint work with M. Krstic and P.A.
%Parrilo~\cite{AAA_MK_PP_CDC11_no_Poly_Lyap}, we recently gave a
%remarkably simple example of a (non-homogeneous) quadratic
%polynomial vector field in two variables that is GAS but does not
%admit a polynomial Lyapunov function (of any degree). An
%independent earlier example that appears in a book by Bacciotti
%and Rosier~\cite[Prop. 5.2]{Bacciotti.Rosier.Liapunov.Book} was
%brought to our attention after our work was submitted. We refer
%the reader to~\cite{AAA_MK_PP_CDC11_no_Poly_Lyap} for a discussion
%on the differences between the two examples, the main one being
%that the example in~\cite{Bacciotti.Rosier.Liapunov.Book} does not
%admit a polynomial Lyapunov function even locally but unlike the
%example in~\cite{AAA_MK_PP_CDC11_no_Poly_Lyap} relies on using
%irrational coefficients.

%The situation for homogeneous polynomial vector fields, however,
%seems to be different. We conjecture that for such systems,
%existence of a homogeneous polynomial Lyapunov function is
%necessary and sufficient for (global) asymptotic stability. The
%reason for this conjecture is that we expect that one should be
%able to approximate a continuously differentiable Lyapunov
%function with a polynomial one on the unit sphere, which by
%homogeneity should be enough to imply the Lyapunov inequalities
%everywhere. A formal treatment of this idea is left for future
%work. 

In this subsection, however, we build on the work of Bacciotti and Rosier~\cite[Prop.
5.2]{Bacciotti.Rosier.Liapunov.Book} to prove that the minimum
degree of a polynomial Lyapunov function for an asymptotically stable homogeneous
vector field can be \emph{arbitrarily large} even when the degree and
dimension are fixed respectively to $3$ and $2$.

\begin{proposition}[{~\cite[Prop.
5.2--a]{Bacciotti.Rosier.Liapunov.Book}}]
\label{prop:Bacciotti.Rosier} Consider the vector field
\begin{equation}\label{eq:Bacciotti.Rosier.f0}
%\dot{x}=\begin{pmatrix} -2\lambda
%x_2(x_1^2+x_2^2)-2x_2(2x_1^2+x_2^2) \\ \ \ 4\lambda
%x_1(x_1^2+x_2^2)+2x_1(2x_1^2+x_2^2)  \end{pmatrix},
\begin{array}{lll}
\dot{x}&=&-2\lambda y(x^2+y^2)-2y(2x^2+y^2) \\
\dot{y}&=&\ \ 4\lambda x(x^2+y^2)+2x(2x^2+y^2)
\end{array}
\end{equation}
parameterized by the scalar $\lambda>0$. For all values of
$\lambda$ the origin is a center for
(\ref{eq:Bacciotti.Rosier.f0}), but for any irrational value of
$\lambda$ there exists no polynomial function $V$ satisfying
$\dot{V}(x,y)=\frac{\partial{V}}{\partial{x}}\dot{x}+\frac{\partial{V}}{\partial{y}}\dot{y}=0.$
\end{proposition}

\begin{theorem}\label{thm:no.finite.bound}
Let $\lambda$ be a positive irrational real number and consider
the following homogeneous cubic vector field parameterized by the
scalar $\theta$: \scalefont{.82}
\begin{equation}\label{eq:a.s.cubic.vec.arbitrary.high.Lyap}
\begin{pmatrix}\dot{x} \\ \dot{y}\end{pmatrix} =\begin{pmatrix}\cos(\theta) & -\sin(\theta)\\\sin(\theta)&\ \ \  \cos(\theta)  \end{pmatrix} \begin{pmatrix}-2\lambda y(x^2+y^2)-2y(2x^2+y^2) \\
\ \ 4\lambda x(x^2+y^2)+2x(2x^2+y^2)
\end{pmatrix}.
\end{equation} \normalsize
Then for any even degree $d$ of a candidate polynomial Lyapunov
function, there exits a $\theta>0$ small enough such that the
vector field in (\ref{eq:a.s.cubic.vec.arbitrary.high.Lyap}) is
asymptotically stable but does not admit a polynomial Lyapunov
function of degree $\leq d$.
\end{theorem}

\begin{proof}
Consider the (non-polynomial) positive definite Lyapunov function
\begin{equation}\nonumber
V(x,y)=(2x^2+y^2)^\lambda(x^2+y^2)
\end{equation}
whose derivative along the trajectories of
(\ref{eq:a.s.cubic.vec.arbitrary.high.Lyap}) is equal to
$$\dot{V}(x,y)=-\sin(\theta)(2x^2+y^2)^{\lambda-1}(\dot{x}^2+\dot{y}^2).$$ Since
$\dot{V}$ is negative definite for $0<\theta<\pi$, it follows that
for $\theta$ in this range, the origin of
(\ref{eq:a.s.cubic.vec.arbitrary.high.Lyap}) is asymptotically
stable.

To establish the claim in the theorem, suppose for the sake of
contradiction that there exists an upper bound $\bar{d}$ such that
for all $0<\theta<\pi$ the system admits a (homogeneous)
polynomial Lyapunov function of degree $d(\theta)$ with
$d(\theta)\leq\bar{d}$. Let $\hat{d}$ be the least common
multiplier of the degrees $d(\theta)$ for $0<\theta<\pi$. (Note
that $d(\theta)$ can at most range over all even positive integers
less than or equal to $\bar{d}$.) Since positive powers of
Lyapunov functions are valid Lyapunov functions, it follows that
for every $0<\theta<\pi$, the system admits a homogeneous
polynomial Lyapunov function $W_\theta$ of degree $\hat{d}$. By
rescaling, we can assume without loss of generality that all
Lyapunov functions $W_\theta$ have unit area on the unit sphere.
Let us now consider the sequence $\{W_\theta\}$ as
$\theta\rightarrow 0$. We think of this sequence as residing in a
compact subset of \aaa{$\mathbb{R}^{\hat{d}+1}$}
associated with the set $P_{2,\hat{d}}$ of (coefficients of) all
nonnegative bivariate homogeneous polynomials of degree $\hat{d}$
with unit area on the unit sphere. Since every bounded sequence
has a converging subsequence, it follows that there must exist a
subsequence of $\{W_\theta\}$ that converges (in the coefficient
sense) to some polynomial $W_0$ belonging to $ P_{2,\hat{d}}$.
Since convergence of this subsequence also implies convergence of
the associated gradient vectors, we get that
$$\dot{W}_0(x,y)=\frac{\partial{W}_0}{\partial{x}}\dot{x}+\frac{\partial{W}_0}{\partial{y}}\dot{y}\leq0.$$
On the other hand, when $\theta=0$, the vector field in
(\ref{eq:a.s.cubic.vec.arbitrary.high.Lyap}) is the same as the
one in (\ref{eq:Bacciotti.Rosier.f0}) and hence the trajectories
starting from any nonzero initial condition go on periodic orbits.
This however implies that $\dot{W}=0$ everywhere and in view of
Proposition~\ref{prop:Bacciotti.Rosier} we have a contradiction.
\end{proof}

\begin{remark}
Unlike the result in~\cite[Prop.
5.2]{Bacciotti.Rosier.Liapunov.Book}, it is easy to establish the
result of Theorem~\ref{thm:no.finite.bound} without having to use
irrational coefficients in the vector field. One approach is to
take an irrational number, e.g. $\pi$, and then think of a
sequence of vector fields given by
(\ref{eq:a.s.cubic.vec.arbitrary.high.Lyap}) that is parameterized
by both $\theta$ and $\lambda$. We let the $k$-th vector field in
the sequence have $\theta_k=\frac{1}{k}$ and $\lambda_k$ equal to
a rational number representing $\pi$ up to $k$ decimal digits.
Since in the limit as $k\rightarrow\infty$ we have
$\theta_k\rightarrow 0$ and $\lambda_k\rightarrow \pi$, it should
be clear from the proof of Theorem~\ref{thm:no.finite.bound} that
for any integer $d$, there exists an AS bivariate homogeneous
cubic vector field with \emph{rational} coefficients that does not
have a polynomial Lyapunov function of degree less than $d$.
\end{remark}

\subsection{Lack of monotonicity in the degree of polynomial Lyapunov
functions}\label{sec:no.monotonicity.in.degree}  %change the label to subsection. be careful with references to it. 

In the last part of this section, we point out yet another difficulty associated with polynomial Lyapunov functions: lack of monotonicity in their degree. \aaa{Among other scenarios, this issue can arise when one seeks homogeneous polynomial Lyapunov functions (e.g., when analyzing stability of homogeneous vector fields). Note that unlike the case of (non-homogeneous) polynomials, the set of homogeneous polynomials of degree $d$ is not a subset of the set of homogeneous polynomials of degree $\geq d$.}

If a dynamical system admits a quadratic Lyapunov function $V$, then it also admits a polynomial Lyapunov function of any higher even
degree (e.g., simply given by $V^k$ for $k=2,3,\ldots$). However,
our next theorem shows that for homogeneous systems that do not
admit a quadratic Lyapunov function, such a monotonicity property
in the degree of polynomial Lyapunov functions may not hold.

\begin{theorem}\label{thm:no.monotonicity}
Consider the following homogeneous cubic vector field
parameterized by the scalar $\theta$:
\begin{equation}\label{eq:non.monotonicity.vec.field}
%\begin{pmatrix}\dot{x} \\ \dot{y}\end{pmatrix} =\begin{pmatrix}\ \ \  \cos(\frac{\pi}{2}+\theta) & \sin(\frac{\pi}{2}+\theta)\\-\sin(\frac{\pi}{2}+\theta)&\cos(\frac{\pi}{2}+\theta)  \end{pmatrix} \begin{pmatrix}x^3 \\
%y^3
%\end{pmatrix}.
\begin{pmatrix}\dot{x} \\ \dot{y}\end{pmatrix} =\aaa{-}\begin{pmatrix}\sin(\theta) & \ \ \aaa{-}\cos(\theta)\\\cos(\theta)&\sin(\theta)  \end{pmatrix} \begin{pmatrix}x^3 \\
y^3
\end{pmatrix}.
\end{equation}
There exists a range of values for the parameter $\theta>0$ for
which the vector field is asymptotically stable, has no
homogeneous polynomial Lyapunov function of degree $6$, but admits
a homogeneous polynomial Lyapunov function of degree $4$.
\end{theorem}

\begin{proof}
Consider the positive definite Lyapunov function
\begin{equation}\label{eq:V=x^4+y^4}
V(x,y)=x^4+y^4.
\end{equation}
The derivative of this Lyapunov function is given by
$$\dot{V}(x,y)=-4\sin(\theta)(x^6+y^6),$$ which is negative definite for $0<\theta<\pi$.
Therefore, when $\theta$ belongs to this range, the origin of
(\ref{eq:a.s.cubic.vec.arbitrary.high.Lyap}) is asymptotically
stable and the system admits the degree $4$ Lyapunov function
given in (\ref{eq:V=x^4+y^4}). On the other hand, we claim that
for $\theta$ small enough, the system \aaa{does not} admit a degree $6$
(homogeneous) polynomial Lyapunov function. To argue by
contradiction, we suppose that for arbitrarily small and positive
values of $\theta$ the system admits sextic Lyapunov functions
$W_\theta$. Since the vector field satisfies the symmetry \aaa{(equivariance),}
\begin{equation}\nonumber
\begin{pmatrix}\dot{x}(y,-x) \\ \dot{y}(y,-x)\end{pmatrix} =\begin{pmatrix}0 & 1\\-1&0  \end{pmatrix} \begin{pmatrix}\dot{x} \\
\dot{y}
\end{pmatrix},
\end{equation}
we can assume that the Lyapunov functions $W_\theta$ satisfy the
symmetry $W_\theta(y,-x)=W_\theta(x,y)$. \footnote{To see this,
note that any Lyapunov function $V_\theta$ for this system can be
made into one satisfying this symmetry by letting
$W_\theta(x,y)=V_\theta(x,y)+V_\theta(y,-x)+V_\theta(-x,-y)+V_\theta(-y,x)$.}
This means that $W_\theta$ can be parameterized with no odd
monomials, i.e., in the form
\begin{equation}\nonumber
W_\theta(x,y)=c_1x^6+c_2x^2y^4+c_3x^4y^2+c_4y^6,
\end{equation}
where the coefficients $c_1,\ldots,c_4$ \aaa{may depend on} $\theta$. Since by our assumption $\dot{W}_\theta$
is negative definite for $\theta$ arbitrarily small, an argument
identical to the one used in the proof of
Theorem~\ref{thm:no.finite.bound} implies that as
$\theta\rightarrow 0$, $W_\theta$ converges to a nonzero sextic
homogeneous polynomial $W_0$ whose derivative $\dot{W}_0$ along
the trajectories of (\ref{eq:non.monotonicity.vec.field}) (with
$\theta=0$) is non-positive. However, note that when $\theta=0$,
the trajectories of (\ref{eq:non.monotonicity.vec.field}) go on
periodic orbits tracing the level sets of the function $x^4+y^4$.
This implies that
$$\dot{W}_0=\frac{\partial{W_0}}{\partial{x}}y^3+\frac{\partial{W_0}}{\partial{y}}(-x^3)=0.$$
If we write out this equation, we obtain \scalefont{.92}
\begin{equation}\nonumber
\dot{W}_0=(6c_1-4c_2)x^5y^3+2c_2xy^7-2c_3x^7y+(4c_3-6c_4)x^3y^5=0,
\end{equation}\normalsize
which implies that $c_1=c_2=c_3=c_4=0$, hence a contradiction.
\end{proof}

\begin{remark}
We have numerically computed the range $0<\theta<\aaa{0.0267}$, for
which the conclusion of Theorem~\ref{thm:no.monotonicity} holds.
This bound has been computed via sum of squares relaxation and
semidefinite programming (SDP) by using the SDP solver
SeDuMi~\cite{sedumi}. What allows the search for a Lyapunov
function for the vector field in
(\ref{eq:non.monotonicity.vec.field}) to be exactly cast as a
semidefinite program is the fact that all nonnegative bivariate
forms are sums of squares.
\end{remark}

\section{(Non)-existence of sum of squares Lyapunov
functions}\label{sec:(non)-existence.sos.lyap}

In this section, we suppose that the polynomial vector field at
hand admits a polynomial Lyapunov function, and we would like to
investigate whether such a Lyapunov function can be found with sos
programming. In other words, we would like to see whether the
constrains in (\ref{eq:V.SOS}) and (\ref{eq:-Vdot.SOS}) are more
conservative than the true Lyapunov inequalities in
(\ref{eq:V.positive}) and (\ref{eq:Vdot.negative}).

%We think of the sos Lyapunov conditions in (\ref{eq:V.SOS}) and
%(\ref{eq:-Vdot.SOS}) as sufficient conditions for the strict
%inequalities in (\ref{eq:V.positive}) and (\ref{eq:Vdot.negative})
%even though sos decomposition in general merely guarantees
%non-strict inequalities. The reason for this is that when an sos
%feasibility problem is strictly feasible, the polynomials returned
%by interior point algorithms are automatically positive definite
%(see~\cite[p. 41]{AAA_MS_Thesis} for more discussion).

In 1888, Hilbert~\cite{Hilbert_1888} showed that for polynomials
in $n$ variables and of degree $d$, the notions of nonnegativity
and sum of squares are equivalent if and only if $n=1$, $d=2$, or
$(n,d)$=$(2,4)$. A homogeneous version of the same result states
that nonnegative homogeneous polynomials in $n$ variables and of
degree $d$ are sums of squares if and only if $n=2$, $d=2$, or
$(n,d)$=$(3,4)$. The first explicit example of a nonnegative
polynomial that is not sos is due to Motzkin~\cite{MotzkinSOS} and
appeared nearly 80 years after the paper of Hilbert; see the
survey in~\cite{Reznick}. Still today, finding examples of such
polynomials is a challenging task, especially if additional
structure is required on the polynomial; see
e.g.~\cite{AAA_PP_not_sos_convex_journal}. This itself is a
premise for the powerfulness of sos techniques at least in low
dimensions and degrees.

\begin{remark}\label{rmk:Lyap.is.different}
Existence of nonnegative polynomials that are not sums of squares
does not imply on its own that the sos conditions in
(\ref{eq:V.SOS}) and (\ref{eq:-Vdot.SOS}) are more conservative
than the Lyapunov inequalities in (\ref{eq:V.positive}) and
(\ref{eq:Vdot.negative}). Since Lyapunov functions are not in
general unique, it could happen that within the set of valid
polynomial Lyapunov functions of a given degree, there is always
at least one that satisfies the sos conditions (\ref{eq:V.SOS})
and (\ref{eq:-Vdot.SOS}). Moreover, many of the known examples of
nonnegative polynomials that are not sos have multiple zeros and
local minima~\cite{Reznick} and therefore cannot serve as Lyapunov
functions. Indeed, if a function has a local minimum other than
the origin, then its value evaluated on a trajectory starting from
the local minimum would not be decreasing.
\end{remark}

\subsection{A motivating example}\label{subsec:motivating.example}

\aaa{The following is a concrete example of the use of sum of squares techniques for finding Lyapunov functions. It will also motivate the type of questions that we would like to study in this section; namely, if sos programming fails to find a polynomial Lyapunov function of a particular degree, then what are the different possibilities for existence of polynomial Lyapunov functions?}

%The following example will help motivate the kind of questions
%that we are addressing in this section.

\begin{example}\label{ex:poly8}
Consider the dynamical system
\begin{equation} \label{eq:poly8}
\begin{array}{lll}
\dot{x_{1}}&=&-0.15x_1^7+200x_1^6x_2-10.5x_1^5x_2^2-807x_1^4x_2^3\\
\ &\ &+14x_1^3x_2^4+600x_1^2x_2^5-3.5x_1x_2^6+9x_2^7
\\
\ &\ & \ \\
\dot{x_{2}}&=&-9x_1^7-3.5x_1^6x_2-600x_1^5x_2^2+14x_1^4x_2^3\\
\ &\ &+807x_1^3x_2^4-10.5x_1^2x_2^5-200x_1x_2^6-0.15x_2^7.
\end{array}
\end{equation}
A typical trajectory of the system that starts from the initial
condition $x_0=(2, 2)^T$ is plotted in Figure~\ref{fig:poly8}. Our
goal is to establish global asymptotic stability of the origin by
searching for a polynomial Lyapunov function. Since the vector
field is homogeneous, the search can be restricted to homogeneous
Lyapunov functions~\cite{Hahn_stability_book},~\cite{HomogHomog}.
To employ the sos technique, we can use the software package
SOSTOOLS~\cite{sostools} to search for a Lyapunov function
satisfying the sos conditions (\ref{eq:V.SOS}) and
(\ref{eq:-Vdot.SOS}). However, if we do this, we will not find any
Lyapunov functions of degree $2$, $4$, or $6$. If needed, a
certificate from the dual semidefinite program can be obtained,
which would prove that no polynomial of degree up to $6$ can
satisfy the sos requirements (\ref{eq:V.SOS}) and
(\ref{eq:-Vdot.SOS}).

At this point we are faced with the following question. Does the
system really not admit a Lyapunov function of degree $6$ that
satisfies the true Lyapunov inequalities in (\ref{eq:V.positive}),
(\ref{eq:Vdot.negative})? Or is the failure due to the fact that
the sos conditions in (\ref{eq:V.SOS}), (\ref{eq:-Vdot.SOS}) are
more conservative?

Note that when searching for a degree $6$ Lyapunov function, the
sos constraint in (\ref{eq:V.SOS}) is requiring a homogeneous
polynomial in $2$ variables and of degree $6$ to be a sum of
squares. The sos condition (\ref{eq:-Vdot.SOS}) on the derivative
is also a condition on a homogeneous polynomial in $2$ variables,
but in this case of degree $12$. (This is easy to see from
$\dot{V}=\langle \nabla V,f \rangle$.) Recall from our earlier discussion that nonnegativity and sum of squares are equivalent notions for homogeneous bivariate polynomials, irrespective of the degree. Hence, we now have a proof that this dynamical system truly does not have a Lyapunov function of degree $6$ (or lower).

This fact is perhaps geometrically intuitive.
Figure~\ref{fig:poly8} shows that the trajectory of this system is
stretching out in $8$ different directions. So, we would expect
the degree of the Lyapunov function to be at least $8$. Indeed,
when we increase the degree of the candidate function to $8$,
SOSTOOLS and the SDP solver SeDuMi~\cite{sedumi} succeed in
finding the following Lyapunov function:
\begin{eqnarray}\nonumber
V(x)&=&0.02x_1^8+0.015x_1^7x_2+1.743x_1^6x_2^2-0.106x_1^5x_2^3
\nonumber \\
 \ &\ &-3.517x_1^4x_2^4+0.106x_1^3x_2^5+1.743x_1^2x_2^6
 \nonumber\\
 \ &\ &-0.015x_1x_2^7+0.02x_2^8. \nonumber
\end{eqnarray}
\begin{figure}%[thpb]
\centering \scalebox{0.38} {\includegraphics{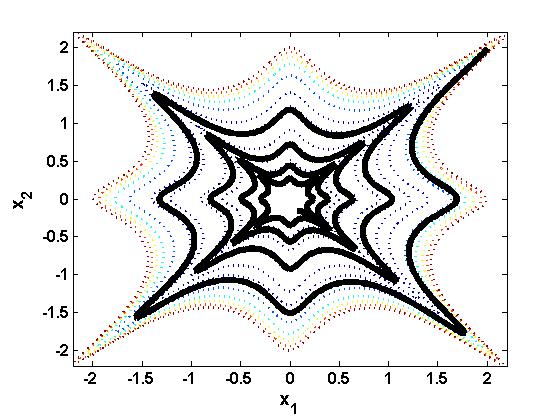}}
\caption{A typical trajectory of the vector filed in
Example~\ref{ex:poly8} (solid), level sets of a degree $8$
polynomial Lyapunov function (dotted).} \label{fig:poly8}
\end{figure}
The level sets of this Lyapunov function are plotted in
Figure~\ref{fig:poly8} and are clearly invariant under the
trajectory. $\triangle$
\end{example}

\subsection{A counterexample \aaa{to existence of sos Lyapunov functions}}\label{subsec:the.counterexample}
Unlike the scenario in the previous example, we now show that a
failure in finding a Lyapunov function of a particular degree via
sum of squares programming can also be due to the gap between
nonnegativity and sum of squares. What will be conservative in the
following counterexample is the sos condition on the
derivative.

%\footnote{This counterexample has appeared in our
%earlier work~\cite{AAA_MS_Thesis} but not with a complete proof.}

%Recall that every time we search for a Lyapunov function using
%sos programming, we use the sos relaxation twice. First for
%nonnegativity of $V$, and second for nonpositivity of $\dot{V}$.
%Each of these relaxations may in general be conservative. Many of
%the well-known examples of nonnegative polynomials that are not
%sos have more than one local minimum~\cite{Reznick}. On the other
%hand, for continuous time systems, Lyapunov functions cannot have
%local minima, except for a single global minimum at the origin. To
%see this, not that that if a trajectory starts exactly at the
%local minima of $V$, irrespective of the direction in which $f$
%moves the trajectory, the Lyapunov function will locally
%increase\footnote{There are Lyapunov functions that are not
%required to decrease monotonically but still prove stability of
%dynamical systems; see~\cite{AAA_MS_Thesis}. These Lyapunov
%functions are allowed to increase locally and therefore can, in
%theory, have local minimas.}. As a consequence, many of the
%well-known nonnegative polynomials that are not sos cannot serve
%as Lyapunov functions. On the other hand, the following example
%shows that an sos relaxation on $\dot{V}$ can be conservative.

Consider the dynamical system
\begin{equation} \label{eq:sos.conservative.dynamics}
\begin{array}{lll}
\dot{x_{1}}&=&-x_1^3x_2^2+2x_1^3x_2-x_1^3+4x_1^2x_2^2-8x_1^2x_2+4x_1^2 \\
\ &\ &-x_1x_2^4+4x_1x_2^3-4x_1+10x_2^2
\\ \ &\ & \ \\
\dot{x_{2}}&=&-9x_1^2x_2+10x_1^2+2x_1x_2^3-8x_1x_2^2-4x_1-x_2^3
\\
\ &\ &+4x_2^2-4x_2.
\end{array}
\end{equation}
One can verify that the origin is the only equilibrium point for
this system, and therefore it makes sense to investigate global
asymptotic stability. If we search for a quadratic Lyapunov
function for (\ref{eq:sos.conservative.dynamics}) using sos
programming, we will not find one. It will turn out that the
corresponding semidefinite program is infeasible. We will prove
shortly why this is the case, i.e, why no quadratic function $V$
can satisfy
\begin{equation}\label{eq:both.sos.conditions}
\begin{array}{rl}
V & \mbox{sos} \\
-\dot{V} & \mbox{sos.}
\end{array}
\end{equation}
Nevertheless, we claim that
\begin{equation}\label{eq:V.0.5x1^2+.5x2^2}
V(x)=\frac{1}{2}x_1^2+\frac{1}{2}x_2^2
\end{equation}
is a valid Lyapunov function. Indeed, one can check that
\begin{equation}\label{eq:Vdot.-Motzkin}
\dot{V}(x)=x_1\dot{x}_1+x_2\dot{x}_2=-M(x_1-1,x_2-1),
\end{equation}
where $M(x_1,x_2)$ is the Motzkin polynomial~\cite{MotzkinSOS}:
\begin{equation}\nonumber%\label{eq:Motzkin}
M(x_1,x_2)=x_1^4x_2^2+x_1^2x_2^4-3x_1^2x_2^2+1.
\end{equation}
This polynomial is just a dehomogenized version of the Motzkin
form presented before, and it has the property of being
nonnegative but not a sum of squares. The polynomial $\dot{V}$ is
strictly negative everywhere, except for the origin and three
other points $(0,2)^{T}$, $(2,0)^{T}$, and $(2,2)^{T}$, where
$\dot{V}$ is zero. However, at each of these three points we have
$\dot{x}\neq0$. Once the trajectory reaches any of these three
points, it will be kicked out to a region where $\dot{V}$ is
strictly negative. Therefore, by LaSalle's invariance principle
(see e.g. \cite[p. 128]{Khalil:3rd.Ed}), the quadratic Lyapunov
function in (\ref{eq:V.0.5x1^2+.5x2^2}) proves global asymptotic
stability of the origin of (\ref{eq:sos.conservative.dynamics}).

\begin{figure}[h]
\begin{center}
    \mbox{
      \subfigure[Shifted Motzkin polynomial is nonnegative but not sos. \aaa{This polynomial is $-\dot{V}$; see (\ref{eq:Vdot.-Motzkin}).}]
      {\label{subfig:shifted.Motzkin}\scalebox{0.28}{\includegraphics{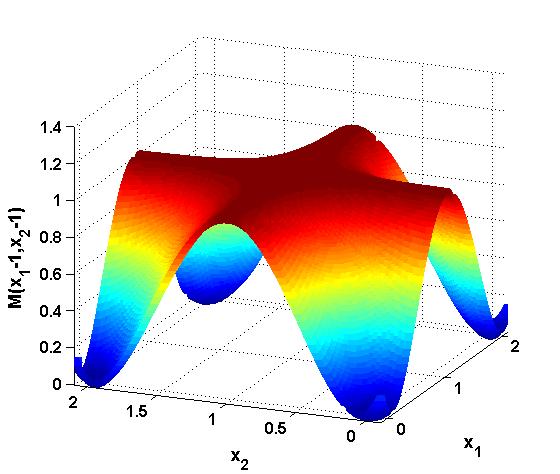}}}}
\mbox{
      \subfigure[Typical trajectories of (\ref{eq:sos.conservative.dynamics}) (solid), level sets of $V$ (dotted).]
      {\label{subfig:sos.hurt.trajec}\scalebox{0.25}{\includegraphics{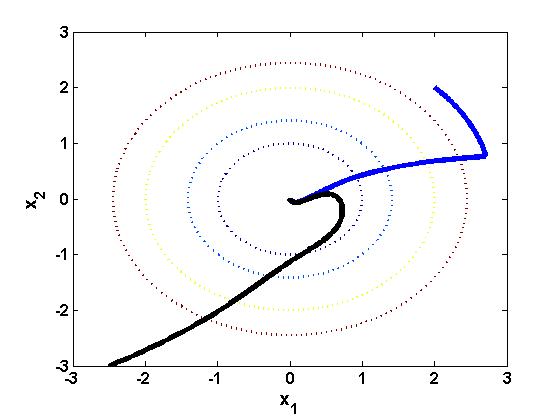}}}

      \subfigure[Level sets of a quartic Lyapunov function found through sos~programming.]
      {\label{subfig:sos.hurt.Lyap4}\scalebox{0.25}{\includegraphics{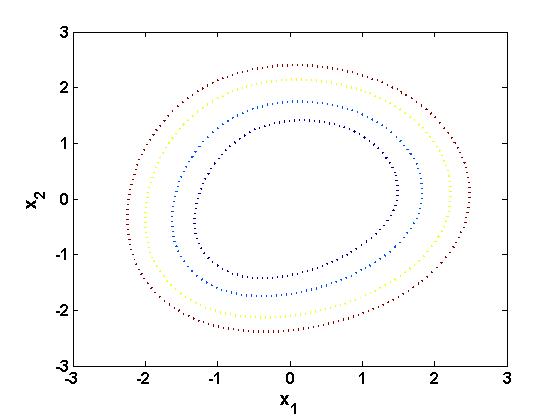}}} }

    \caption{The quadratic polynomial $\frac{1}{2}x_1^2+\frac{1}{2}x_2^2$ is a valid Lyapunov function for the vector field in (\ref{eq:sos.conservative.dynamics}) but it is not detected through sos programming.}
\label{fig:sos.hurt.Motzkin}
\end{center}
\end{figure}

The fact that $\dot{V}$ is zero at three points other than the
origin is not the reason why sos programming is failing. After
all, when we impose the condition that $-\dot{V}$ should be sos,
we allow for the possibility of a non-strict inequality. The
reason why our sos program does not recognize
(\ref{eq:V.0.5x1^2+.5x2^2}) as a Lyapunov function is that the
shifted Motzkin polynomial in (\ref{eq:Vdot.-Motzkin}) is
nonnegative but it is not a sum of squares. This sextic polynomial
is plotted in Figure~\ref{subfig:shifted.Motzkin}. Trajectories of
(\ref{eq:sos.conservative.dynamics}) starting at $(2,2)^{T}$ and
$(-2.5,-3)^{T}$ along with level sets of $V$ are shown in
Figure~\ref{subfig:sos.hurt.trajec}.

So far, we have shown that $V$ in (\ref{eq:V.0.5x1^2+.5x2^2}) is a
valid Lyapunov function but does not satisfy the sos conditions in
(\ref{eq:both.sos.conditions}). We still need to show why no other
quadratic Lyapunov function
\begin{equation}\label{eq:U(x).general.quadratic}
U(x)=c_1x_1^2+c_2x_1x_2+c_3x_2^2
\end{equation}
can satisfy the sos conditions either.\footnote{Since we can
assume that the Lyapunov function $U$ and its gradient vanish at
the origin, linear or constant terms are not needed in
(\ref{eq:U(x).general.quadratic}).} We will in fact prove the
stronger statement that $V$ in (\ref{eq:V.0.5x1^2+.5x2^2}) is the
only valid quadratic Lyapunov function for this system up to
scaling, i.e., any quadratic function $U$ that is not a scalar
multiple of $\frac{1}{2}x_1^2+\frac{1}{2}x_2^2$ cannot satisfy
$U\geq0$ and $-\dot{U}\geq0$. It will even be the case that no
such $U$ can satisfy $-\dot{U}\geq0$ alone. (The latter fact is to
be expected since global asymptotic stability of
(\ref{eq:sos.conservative.dynamics}) together with $-\dot{U}\geq0$
would automatically imply $U\geq0$; see~\cite[Theorem
1.1]{AAA_PP_ACC11_Lyap_High_Deriv}.)

So, let us show that $-\dot{U}\geq0$ implies $U$ is a scalar
multiple of $\frac{1}{2}x_1^2+\frac{1}{2}x_2^2$. Because Lyapunov
functions are closed under positive scalings, without loss of
generality we can take $c_1~=~1$. One can check that
$$-\dot{U}(0,2)=-80c_2,$$ so to have $-\dot{U}\geq0$, we need
$c_2\leq0$. Similarly, $$-\dot{U}(2,2)=-288c_1+288c_3,$$ which
implies that $c_3\geq1$. Let us now look at
\begin{equation}\label{eq:-Udot(x1,1)}
\begin{array}{lll}
-\dot{U}(x_1,1)&=& -c_2x_1^3 + 10c_2x_1^2 + 2c_2x_1 - 10c_2 -
2c_3x_1^2\\
\ &\ & + 20c_3x_1 + 2c_3 + 2x_1^2 - 20x_1.
\end{array}
\end{equation}
If we let $x_1\rightarrow -\infty$, the term $-c_2x_1^3$ dominates
this polynomial. Since $c_2\leq0$ and $-\dot{U}\geq0$, we conclude
that $c_2=0$. Once $c_2$ is set to zero in (\ref{eq:-Udot(x1,1)}),
the dominating term for $x_1$ large will be $(2-2c_3)x_1^2$.
Therefore to have $-\dot{U}(x_1,1)\geq0$ as
$x_1\rightarrow\pm\infty$ we must have $c_3\leq1$. Hence, we
conclude that $c_1=1, c_2=0, c_3=1$, and this finishes the proof.

Even though sos programming failed to prove stability of the
system in (\ref{eq:sos.conservative.dynamics}) with a quadratic
Lyapunov function, if we increase the degree of the candidate
Lyapunov function from $2$ to $4$, then SOSTOOLS succeeds in
finding a quartic Lyapunov function
\begin{eqnarray}\nonumber %\label{eq:sos.hurt.lyap.fn.poly4}
W(x)&=&0.08x_1^4-0.04x_1^3+0.13x_1^2x_2^2+0.03x_1^2x_2
\nonumber \\
 \ &\
 &+0.13x_1^2+0.04x_1x_2^2-0.15x_1x_2\nonumber
 \\
 \ &\ &+0.07x_2^4-0.01x_2^3+0.12x_2^2, \nonumber
\end{eqnarray}
which satisfies the sos conditions in
(\ref{eq:both.sos.conditions}). The level sets of this function
are close to circles and are plotted in
Figure~\ref{subfig:sos.hurt.Lyap4}.

Motivated by this example, it is natural to ask whether it is
always true that upon increasing the degree of the Lyapunov
function one will find Lyapunov functions that satisfy the sum of
squares conditions in (\ref{eq:both.sos.conditions}). In the next
subsection, we will prove that this is indeed the case, at least
for planar systems such as the one in this example, and also for
systems that are homogeneous.

%statement is indeed true at least for homogeneous vector fields.

%We should mention that the example we just described was contrived
%to make our point. For many practical problems, sos programming
%has provably shown to be a powerful technique~\cite{PabloLyap},
%\cite{PapP02}, \cite{PraP03}. There are some recent
%results~\cite{Blekhermansos}, however, that show for a fixed
%degree, as the dimension goes up the gap between nonnegativity and
%sos broadens. The extent to which this can impact existence of
%Lyapunov functions in higher dimensions is yet to be investigated.

\subsection{Converse sos Lyapunov
theorems}\label{subsec:converse.sos.results}

In~\cite{Peet.Antonis.converse.sos.CDC},
\cite{Peet.Antonis.converse.sos.journal}, it is shown that if a
system admits a polynomial Lyapunov function, then it also admits
one that is a sum of squares. However, the results there do not
lead to any conclusions as to whether the negative of the
derivative of the Lyapunov function is sos, i.e, whether condition
(\ref{eq:-Vdot.SOS}) is satisfied. As we remarked before, there is
therefore no guarantee that the semidefinite program can find such
a Lyapunov function. Indeed, our counterexample in the previous
subsection demonstrated this very phenomenon.

The proof technique used
in~\cite{Peet.Antonis.converse.sos.CDC},\cite{Peet.Antonis.converse.sos.journal}
is based on approximating the solution map using the Picard
iteration and is interesting in itself, though the actual
conclusion that a Lyapunov function that is sos exists has a far
simpler proof which we give in the next lemma.

\begin{lemma}\label{lem:W=V^2}
If a polynomial dynamical system has a positive definite
polynomial Lyapunov function $V$ with a negative definite
derivative $\dot{V}$, then it also admits a positive definite
polynomial Lyapunov function $W$ which is a sum of squares.
\end{lemma}
\begin{proof}
Take $W=V^2$. The negative of the derivative $-\dot{W}=-2V\dot{V}$
is clearly positive definite (though it may not be sos).
\end{proof}

We will next prove a converse sos Lyapunov theorem that guarantees
the derivative of the Lyapunov function will also satisfy the sos
condition, though this result is restricted to homogeneous
systems. The proof of this theorem relies on the following
Positivstellensatz result due to Scheiderer.

\begin{theorem}[Scheiderer,~\cite{Claus_Hilbert17}] \label{thm:claus}
Given any two positive definite homogeneous polynomials $p$ and
$q$, there exists an integer $k$ such that $pq^k$ is a sum of
squares.
\end{theorem}

\begin{theorem}\label{thm:poly.lyap.then.sos.lyap}
Given a homogeneous polynomial vector field, suppose there exists
a homogeneous polynomial Lyapunov function $V$ such that $V$ and
$-\dot{V}$ are positive definite. Then, there also exists a
homogeneous polynomial Lyapunov function $W$ such that $W$ is sos
and $-\dot{W}$ is sos.
\end{theorem}

\begin{proof}
Observe that $V^2$ and $-2V\dot{V}$ are both positive definite and
homogeneous polynomials. Applying Theorem~\ref{thm:claus} to these
two polynomials, we conclude the existence of an integer $k$ such
that $(-2V\dot{V})(V^2)^k$ is sos. Let $$W=V^{2k+2}.$$ Then, $W$
is clearly sos since it is a perfect even power. Moreover,
$$-\dot{W}=-(2k+2)V^{2k+1}\dot{V}=-(k+1)2V^{2k}V\dot{V}$$
is also sos by the previous claim.\footnote{Note that $W$
constructed in this proof proves GAS since $-\dot{W}$ is positive
definite and $W$ itself being homogeneous and positive definite is
automatically radially unbounded.}
\end{proof}

\aaa{Finally}, we develop a similar theorem that removes the homogeneity
assumption from the vector field, but instead is restricted to
vector fields on the plane. For this, we need another result of
Scheiderer.

\begin{theorem}[{Scheiderer,~\cite[Cor. 3.12]{Claus_3vars_sos}}] \label{thm:claus.3vars}
Let $p\mathrel{\mathop:}=p(x_1,x_2,x_3)$ and
$q\mathrel{\mathop:}=q(x_1,x_2,x_3)$ be two homogeneous
polynomials in three variables, with $p$ positive semidefinite and
$q$ positive definite. Then, there exists an integer $k$ such that
$pq^k$ is a sum of squares.
\end{theorem}

\begin{theorem}\label{thm:poly.lyap.then.sos.lyap.PLANAR}
Given a (not necessarily homogeneous) polynomial vector field in
two variables, suppose there exists a positive definite polynomial
Lyapunov function $V,$ with $-\dot{V}$ positive definite, and such
that the highest \aaa{degree homogeneous component} of $V$ has no zeros\footnote{This
requirement is only slightly stronger than the requirement of
radial unboundedness, which is imposed on $V$ by Lyapunov's
theorem anyway.}. Then, there also exists a polynomial Lyapunov
function $W$ such that $W$ is sos and $-\dot{W}$ is sos.
\end{theorem}

\begin{proof}
Let $\tilde{V}=V+1$. So, $\dot{\tilde{V}}=\dot{V}$. Consider the
(non-homogeneous) polynomials $\tilde{V}^2$ and
$-2\tilde{V}\dot{\tilde{V}}$ in the variables
$x\mathrel{\mathop:}=(x_1,x_2)$. Let us denote the (even) degrees
of these polynomials respectively by $d_1$ and $d_2$. Note that
$\tilde{V}^2$ is nowhere zero and $-2\tilde{V}\dot{\tilde{V}}$ is
only zero at the origin. Our first step is to homogenize these
polynomials by introducing a new variable $y$. Observing that the
homogenization of products of polynomials equals the product of
homogenizations, we obtain the following two trivariate forms:
\begin{equation}\label{eq:V^2.homoegenized}
y^{2d_1}\tilde{V}^2(\textstyle{\frac{x}{y}}),
\end{equation}
\begin{equation}\label{eq:-2V.Vdot.homogenized}
-2y^{d_1}y^{d_2}\tilde{V}(\textstyle{\frac{x}{y}})\dot{\tilde{V}}(\textstyle{\frac{x}{y}}).
\end{equation}
Since by assumption the highest order term of $V$ has no zeros,
the form in (\ref{eq:V^2.homoegenized}) is positive definite. The
form in (\ref{eq:-2V.Vdot.homogenized}), however, is only positive
semidefinite. In particular, since $\dot{\tilde{V}}=\dot{V}$ has
to vanish at the origin, the form in
(\ref{eq:-2V.Vdot.homogenized}) has a zero at the point
$(x_1,x_2,y)=(0,0,1)$. Nevertheless, since
Theorem~\ref{thm:claus.3vars} allows for positive semidefiniteness
of one of the two forms, by applying it to the forms in
(\ref{eq:V^2.homoegenized}) and (\ref{eq:-2V.Vdot.homogenized}),
we conclude that there exists an integer $k$ such that
\begin{equation}\label{eq:-2V.Vdot.homog.*.V^2.homog^k}
-2y^{d_1(2k+1)}y^{d_2}\tilde{V}(\textstyle{\frac{x}{y}})\dot{\tilde{V}}(\textstyle{\frac{x}{y}})\tilde{V}^{2k}(\textstyle{\frac{x}{y}})
\end{equation}
is sos. Let $W=\tilde{V}^{2k+2}.$ Then, $W$ is clearly sos.
Moreover,
$$-\dot{W}=-(2k+2)\tilde{V}^{2k+1}\dot{\tilde{V}}=-(k+1)2\tilde{V}^{2k}\tilde{V}\dot{\tilde{V}}$$
is also sos because this polynomial is obtained from
(\ref{eq:-2V.Vdot.homog.*.V^2.homog^k}) by setting
$y=1$.\footnote{Once again, we note that the function $W$
constructed in this proof is radially unbounded, achieves its
global minimum at the origin, and has $-\dot{W}$ positive
definite. Therefore, $W$ proves global asymptotic stability.}
\end{proof}

\section{\aaa{Summary and} some open questions}\label{sec:summary.future.work}

\aaa{We studied the basic problem of testing stability of equilibrium points of polynomial differential equations and asked some basic questions: What is the computational complexity of this problem? What kind of ``converse Lyapunov theorems'' can we expect to establish on existence of polynomial Lyapunov functions, upper bounds on their degrees, and guaranteed success of techniques based on sum of squares relaxation and semidefinite programming for finding these Lyapunov functions? Our contributions to these questions are listed in Subsection~\ref{subsec:contributions}.

Some problems of interest that our work leaves open are listed below.}

\paragraph{Open questions regarding complexity.} Of course, the most interesting problem here is to formally answer the questions of Arnold
on undecidability of determining stability for polynomial vector fields. As far as NP-hardness is concerned, our work leaves open the question of establishing NP-hardness of testing asymptotic stability for \emph{quadratic} vector fields. (Recall that such a result cannot be restricted to homogeneous quadratic vector fields, but it is quite likely that the problem for all quadratic vector fields is hard.) \aaa{The complexity of all problems considered in Theorem~\ref{thm:poly.hardness.results} is also open for quadratic vector fields. In general, one can reduce the degree of any vector field to two by introducing polynomially many new variables (see~\cite{Undecidability_vec_fields_survey}). However, this operation may or may not preserve the property of the vector field which is of interest.}

\paragraph{Open questions regarding existence of (sos) polynomial Lyapunov functions.} Mark
Tobenkin asked whether globally exponentially stable polynomial
vector fields always admit polynomial Lyapunov functions. Our
counterexample with Krstic in Section~\ref{sec:no.poly.Lyap}, though GAS and
locally exponentially stable, is not globally exponentially stable
because of exponential growth rates in the large. The
counterexample of Bacciotti and Rosier
in~\cite{Bacciotti.Rosier.Liapunov.Book} is not even locally
exponentially stable. Another \aaa{problem left open} is to prove our conjecture that
GAS homogeneous polynomial vector fields admit homogeneous
polynomial Lyapunov functions. This, together with
Theorem~\ref{thm:poly.lyap.then.sos.lyap}, would imply that
asymptotic stability of homogeneous polynomial systems can always
be decided via sum of squares programming. Also, it is not clear
to us whether the assumption of homogeneity and planarity can be
removed from Theorems~\ref{thm:poly.lyap.then.sos.lyap}
and~\ref{thm:poly.lyap.then.sos.lyap.PLANAR} on existence of sos
Lyapunov functions. Finally, another research direction would be
to obtain upper bounds on the degree of polynomial Lyapunov functions when they do exist. Recall that our Theorem~\ref{thm:no.finite.bound} has already established that bounds depending only on dimension and degree of the vector field are impossible. So the questions is whether one can derive bounds that are computable from the coefficients of the vector field. Some degree bounds are known for Lyapunov analysis of locally exponentially stable systems~\cite{Peet.Antonis.converse.sos.journal}, but they depend on uncomputable properties of the solution such as convergence
rate. As far as sos Lyapunov functions are concerned, degree bounds on Positivstellensatz result of the type in Theorems~\ref{thm:claus} and~\ref{thm:claus.3vars} are known, but typically exponential in size and not very encouraging for practical purposes.

\section{Acknowledgements}
We are grateful to Claus Scheiderer for very helpful discussions around his Positivstellensatz results~\cite{Claus_Hilbert17},~\cite{Claus_3vars_sos}. He was kind \aaa{enough} to write down his unpublished result that we needed~\cite{Claus_Hilbert17}---a reference which now includes \aaa{many} other contributions. \aaa{Amir Ali Ahmadi is grateful to members of the MIT Robot Locomotion group, in particular to Anirudha Majumdar and Russ Tedrake for their contributions to some of the complexity results of this paper, and to Mark Tobenkin for inspiring discussions.}

%Amir Ali Ahmadi is grateful to Anirudha Majumdar, Russ Tedrake, and Mark Tobenkin from the MIT Robot Locomotion %Group for several inspiring discussions.

\bibliographystyle{abbrv}
\bibliography{pablo_amirali}

\end{document}